\documentclass{article}
\usepackage[a4paper, margin=3.5cm]{geometry}
\usepackage[dvipsnames]{xcolor}
\usepackage{graphicx}
\usepackage{amsthm, amsmath, amsfonts, amssymb, mathtools, stmaryrd, dsfont, tikz}
\usepackage{mathrsfs}
\usepackage[scr=boondoxo]{mathalfa}
\usepackage{hyperref}
\hypersetup{colorlinks=true, linkcolor=blue}

\newcommand{\R}{\mathbb{R}}
\newcommand{\C}{\mathbb{C}}
\newcommand{\N}{\mathbb{N}}
\newcommand{\bord}{\partial\Omega}
\newcommand{\nbord}{\R^n\backslash\partial\Omega}
\newcommand{\Bcal}{\mathcal{B}}
\newcommand{\Scal}{\mathcal{S}}
\newcommand{\Dcal}{\mathcal{D}}
\newcommand{\Kcal}{\mathcal{K}}
\newcommand{\Vcal}{\mathcal{V}}
\newcommand{\Wcal}{\mathcal{W}}
\newcommand{\Mcal}{\mathcal{M}}
\newcommand{\Ccal}{\mathcal{C}}
\newcommand{\Tr}{\mathrm T\!\operatorname{r}}

\newcommand{\Hp}{\Bcal(\bord)}
\newcommand{\Hm}{\Bcal'(\bord)}
\newcommand{\ddn}[1]{\frac{\partial #1}{\partial\nu}}
\newcommand{\ddny}[2]{\frac{\partial_{#2} #1}{\partial\nu}}
\newcommand{\dx}{\mathrm dx}
\newcommand{\dy}{\mathrm dy}
\newcommand{\dsig}{\mathrm d\sigma}
\newcommand{\dth}{\mathrm d\theta}
\newcommand\locb{\mathrm{loc_b}}

\theoremstyle{plain}
\newtheorem{theorem}{Theorem}[section]
\newtheorem{lemma}[theorem]{Lemma}
\newtheorem{proposition}[theorem]{Proposition}
\newtheorem{corollary}[theorem]{Corollary}
\newtheorem{definition}[theorem]{Definition}

\theoremstyle{definition}
\newtheorem{remark}[theorem]{Remark}

\providecommand{\keywords}[1]{\noindent\textbf{Keywords:} #1}
\providecommand{\MSC}[1]{\noindent\textbf{2020 MSC:} #1}

\title{Helmholtz transmission problem and intrinsic impedance scattering problem on extension domains}
\author{Gabriel Claret\thanks{Université Paris-Saclay, CentraleSupélec, Fédération de Mathématiques (CNRS FRA-3487) and Laboratoire MICS, Gif-sur-Yvette, France. \texttt{gabriel.claret@centralesupelec.fr}}}
\date{\today}

\begin{document}

\maketitle

\begin{abstract}
We consider a transmission problem for the Helmholtz equation with a fixed, positive wavenumber across the boundary of an extension domain.
Such a boundary can be Lipschitz, fractal, or of varying Hausdorff dimension for instance.
We generalise the notions of layer potential and Neumann-Poincar\'e operators, and of Calder\'on projectors in that context.
Those boundary operators allow to connect the transmission problem (on the whole space) to one-sided problems -- notably, scattering problems -- with Dirichlet, Neumann and Robin boundary conditions, and restate their well-posedness as boundary equations.
Since an extension domain needs no specific boundary measure, the Robin (impedance) condition is not understood in a boundary $L^2$-type space, rather by duality on the trace space itself.
We discuss the well-posedness of the impedance scattering problem in that framework and compare it to the classical $L^2$ setting.
Our analysis allows to generalise optimisation results for acoustic scattering when the obstacle is an extension domain in any dimension.
\end{abstract}

\keywords{Acoustic scattering, Sobolev extension domains, transmission problem, Robin boundary condition, layer potential operators.}

\MSC{Primary 31E05, 35J25; Secondary 35J20, 35J05, 31B10}

\tableofcontents

\section{Introduction}

From sonars to medical imaging or even robot vacuum cleaners, acoustic scattering is at the core of numerous applications.
In an acoustic scattering problem, an incident field (solution to the Helmholtz equation) is scattered by an obstacle in such a way that the total field -- superposition of the incident and scattered ones -- satisfies a given boundary condition.
Modelling that problem, a crucial aspect is the geometry of the obstacle, namely, its regularity.
While there is abundant literature in the case of a smooth (or at least Lipschitz) scatterer~\cite{mitzner_numerical_1967,shaw_transient_1972,bennett_time_1981,li_identifying_2025}, it is known that irregular shapes such as fractals are most relevant to represent the behaviour of real-life objects~\cite{mandelbrot_how_1967,mandelbrot_fractal_1983}.

In this paper, we study the transmission problem for the Helmholtz equation with a positive wavenumber across the boundary of an extension domain, thus generalising the notions of layer potential and Neumann-Poincar\'e operators, and Calder\'on projectors in that context.
The latter allow to link the transmission problem (see~\eqref{Eq:Helm-Prob}), which is two-sided, to one-sided problems with Dirichlet, Neumann or Robin boundary conditions.
In particular, we use our analysis of the transmission problem to consider scattering problems, which have been studied extensively and in a variety of contexts, see for instance~\cite{angell_three_1982, ha-duong_transient_1990, wendland_hypersingular_1990,mclean_strongly_2000, bonazzoli_multi-domain_2024}.
The connection between the transmission and the scattering problems allows to derive the well-known boundary integral equations which can be found in~\cite{colton_integral_2013,colton_inverse_2013} in the case of a $C^2$ obstacle,~\cite{brakhage_uber_1965,chandler-wilde_numerical-asymptotic_2012,chandler-wilde_spectrum_2023} in the case of a Lipschitz obstacle,~\cite{chandler-wilde_sobolev_2017,chandler-wilde_boundary_2021,bannister_acoustic_2022,caetano_hausdorff-measure_2024} when the obstacle is a fractal screen,~\cite{bannister_acoustic_2026} for domains with fractal boundaries and~\cite{caetano_integral_2024,chandler-wilde_integral_2026} in the case of a sound-soft compact scatterer.
In the framework of extension domains, we derive boundary equations similar to those boundary integral equations in the case of Dirichlet (sound-soft), Neumann (sound-hard) and Robin (impedance) boundary conditions, and discuss their unique solvability.

The layer potential operators are an essential tool in the study of transmission problems, in which a solution is sought from the knowledge of its behaviour upon crossing a boundary -- more specifically its jumps in trace and normal derivative.
In classical potential theory on smooth domains, those operators are defined as boundary integrals involving fundamental solutions~\cite{vladimirov_equations_1971}.
Extension domains, however, have no specific boundary measure~\cite{hajlasz_sobolev_2008,jones_quasiconformal_1981}, which prevents from using a similar definition in that framework.
Instead, using a variational approach similar to~\cite{barton_layer_2017}, the layer potential operators were defined in the case of the transmission problem associated to $-\Delta$ and $-\Delta+1$ for extension domains~\cite{claret_layer_2026}.
The main difference between the latter case and the Helmholtz equation is the non-integrability of the solutions at infinity.
Nonetheless, truncation methods as in~\cite{bardos_variational_1994,nedelec_acoustic_2011} allow to state a bounded variational problem equivalent to the transmission problem and to which the standard methods can be applied.

In addition to the transmission problem for $-\Delta$ and $-\Delta+1$ mentioned above, dealing with problems set on extension domains by means of variational methods has proved successful in various contexts: for the study of the Riemann-Hilbert problem~\cite{claret_convergence_2025} or of the Calder\'on problem~\cite{claret_poincare-steklov_2025} for instance.
In these cases, the boundary conditions were always of Dirichlet or Neumann type, respectively set on the trace space or its dual space.
Indeed, the lack of a boundary measure prevents from considering the boundary $L^2$ spaces frequently used in the classical case, see for instance~\cite{arendt_dirichlet--neumann_2015, verchota_layer_1984}.
As a continuation, it appears natural to consider problems in which the boundary of an extension domain is endowed with a Robin (impedance) condition.
To circumvent the lack of a boundary measure -- hence, of a boundary $L^2$ space -- inherent to the setting of extension domains, that condition is to be understood in the dual of the trace space, through a Riesz isometry (see~\eqref{Eq:BC-Rob}).
Doing so, the Robin condition is not understood by means of an additional boundary measure and is, in that sense, intrinsic.
While there are examples of such intrinsic Robin boundary conditions in the literature as in~\cite{bonnet-bendhia_acoustic_2002,becache_perfectly_2003,arendt_extensions_2022}, those papers always consider the condition in the Lipschitz case and with a specific impedance coefficient.
The case of possibly non-Lipschitz domains with Robin boundary condition was studied in~\cite{daners_robin_2000,dong_conormal_2021}, however the boundaries considered were always the support of a measure of dimension $n-1$; we also refer to~\cite{bass_robin_2008} for a probabilistic perspective.
In this work, we consider the Helmholtz problem on the exterior of an extension domain, with a general measure-free Robin boundary condition~\eqref{Eq:Rob-Prob}.
There is a fundamental difference between our approach, relying solely on the trace space, and the standard $L^2$ approach: in our setting, the trace operator is no longer compact.
Consequently, in the Fredholm decomposition of the variational formulation of the problem, the boundary term arising from the Robin condition can no longer be in the compact part, and must necessarily be associated with the coercive part.
To account for this change, new conditions must be met for the problem to be well-posed (Proposition~\ref{Prop:Rob-WP}).

As an application of our analysis on the impedance scattering problem, we consider an optimisation problem in which we seek the optimal impedance at the boundary of an obstacle to maximise the far-field power of a scattered field in a given direction, see Figure~\ref{Fig:Optim}.
That problem was first studied in~\cite{kirsch_optimal_1981} in the case of a planar piecewise $C^2$ domain, and later on generalised to $C^2$ domains in~\cite{angell_generalized_1982}, see also~\cite{colton_integral_2013}.
Those works conclude to the existence of an optimal impedance in a bounded convex or compact class of either $L^2$ or $L^\infty$ impedances at the boundary.
In the spirit of those papers, we solve the optimisation problem in the case of an extension domain as obstacle, in general dimension.

The structure and main results of this paper are as follows: in Section~\ref{Sec:Framework}, we introduce the functional framework in which this study is carried out.
We define the class of (two-sided) admissible domains which, in addition to Lipschitz domains, contains fractal and multifractal domains.
In particular, those domains are not endowed with a specific boundary measure.
The associated notion of trace at the boundary is specified, and that of weak normal derivative is generalised to locally integrable functions.
We prove Green's formula in that context (Theorem~\ref{Th:Green}).
In Section~\ref{Sec:Tr-Prob}, we consider the Helmholtz transmission problem on two-sided admissible domains.
Using variational methods, we define the corresponding layer potential operators in Subsection~\ref{Subsec:LP}, thus yielding Green's third identity for that problem (Theorem~\ref{Th:Green3}).
We prove the layer potentials defined in this measure-free way are connected to the Newtonian kernel, as in the Lipchitz case.
Then, we connect the transmission problem to the interior and exterior (scattering) Helmholtz problems with Dirichlet and Neumann boundary conditions in Subsection~\ref{Subsec:NP-Op} by means of the Neumann-Poincaré operators and Calder\'on projectors.
In Section~\ref{Sec:Scattering}, we turn to the impedance scattering problem by a two-sided admissible obstacle.
We introduce a measure-free, intrinsic Robin boundary condition and discuss the well-posedness of the exterior Robin problem in that sense in Subsection~\ref{Subsec:Robin}.
Relying on the connection between the transmission problem and the scattering problem, we express the well-posedness of the latter Robin problem in terms of a boundary equation (Theorem~\ref{Th:Boundary-Integral-Equation}), in the spirit of the boundary integral equations for more regular domains (which are endowed with a boundary measure).
We compare our intrinsic approach to the classical one, that is by means of the space $L^2$ when a boundary measure is specified.
Finally, we consider an optimisation problem in Subsection~\ref{Subsec:Control}, consisting in finding an optimal impedance boundary condition on a given obstacle to maximise the energy of the scattered field in a certain direction at infinity.

\paragraph{Notations.}
Throughout this paper, $n\ge2$ denotes the dimension of the ambient space.
A set is referred to as a domain if it is nonempty, open and connected.
All domains considered are equal to the interiors of their closures.
If $A$ is a subset of $\R^n$, we will denote by $A^c\coloneq\R^n\backslash A$ its complement, $\mathring A$ its interior and $\overline A$ its closure.
We denote by $\mathscr D(A)$ the space of compactly supported infinitely differentiable functions on $A$, and by $\mathscr D'(A)$ its dual space (the space of distributions on $A$).
For $x\in\R^n$ and $r>0$, the $n$-dimensional open Euclidean ball of center $x$ and radius $r$ will be denoted by $B_r(x)$.
The imaginary unit will be denoted by $\mathrm i$ ($\mathrm i^2=-1$).
If $H$ and $V$ are Hilbert spaces, then $\mathcal{L}(H,V)$ denotes the set of all linear continuous maps from $H$ to $V$, and $\mathcal{L}(H)\coloneq\mathcal{L}(H,H)$.
$u\oplus v$ is the unique element of the direct sum $H\oplus V$ with $u\in H$ and $v\in V$.
All Sobolev spaces are understood as complex-valued.
If $A$ is an open set of $\R^n$, we denote by $L^2_{\locb}(A)$ the set of all $u$ which are $L^2$ on every bounded subset of $A$ (that is, the counterpart of $L^2_{\mathrm{loc}}(A)$ with respect to bounded sets instead of compact sets, ensuring integrability up to the boundary).
Similarly, we denote by $H^1_{\mathrm{loc_b}}(A)$ the set of all $u$ which are $H^1$ on every bounded subset of $A$.
Asymptotic comparisons will be expressed using Landau notations.
In particular, if $f,g:\C\to\C$ are such that $|f|=\varepsilon|g|$ where $\varepsilon:\R\to\R$ is bounded (resp. goes to $0$) as $|x|$ goes to $\infty$, we will denote $f(x)=\underset{|x|\to\infty}{O}(g(x))$ (resp. $f(x)=\underset{|x|\to\infty}{o}(g(x))$).

\section{Functional framework}\label{Sec:Framework}

We begin with the introduction of the class of domains on which this work will be carried out.

\begin{definition}[Admissible domain]\label{Def:Admissible-Dom}
A domain $\Omega$ of $\R^n$ is said to be an admissible domain if
\begin{enumerate}
\item[(i)] $\Omega$ is bounded;
\item[(ii)] $\Omega$ is an $H^1$-extension domain~\cite{hajlasz_sobolev_2008}: there exists a bounded, linear operator $\operatorname{Ext}_\Omega:H^1(\Omega)\to H^1(\R^n)$ such that
\begin{equation*}
\forall u\in H^1(\Omega),\quad (\operatorname{Ext}_\Omega u)|_\Omega=u.
\end{equation*}
\end{enumerate}
It is said to be two-sided admissible if, in addition,
\begin{enumerate}
\item[(iii)] $\overline{\Omega}^c$ is an $H^1$-extension domain.
\end{enumerate}
\end{definition}

\begin{remark}
While Definition~\ref{Def:Admissible-Dom} mentions a \emph{domain} (thus, a connected set), the analysis carried out in this paper can also apply to a finite union of mutually disjoint $H^1$-admissible domains.
For simplicity, we will keep assuming that $\Omega$ is connected.
\end{remark}

Condition (i) implies that $\bord$ has positive $H^1$-capacity~\cite[Theorem 5.1.13]{adams_function_1996}, defined for any $A\subset\R^n$ by
\begin{equation*}
\operatorname{Cap}(A)\coloneq\inf\big\{\|u\|_{H^1(\R^n)}^2\;\big|\;u\ge 1 \mbox{ a.e. on a neighbourhood of }A\big\},
\end{equation*}
with the convention $\|u\|_{H^1(\R^n)}=+\infty$ if $u\notin H^1(\R^n)$, see~\cite[Section 2.1]{fukushima_dirichlet_2010},~\cite[Section 7.2]{mazya_boundary_1991} and~\cite[Section 2]{biegert_traces_2009}.
In what follows, any mention of capacity will refer to that one, and so will the notions `quasi-everywhere' (q.e.) and `quasi-continuous'.
It follows from (ii) that $\bord$ is a null-set with respect to Lebesgue's measure on $\R^n$~\cite[p.~2]{hajlasz_sobolev_2008}.

The class of admissible domains contains that of Lipschitz domains, as well as domains with fractal boundaries and boundaries with changing Hausdorff dimension. It is known~\cite[Theorem 2]{hajlasz_sobolev_2008} that any $H^1$-extension domain is an $n$-set, meaning that it satisfies
\begin{equation*}
\exists\, c>0,\;\forall x\in\Omega,\;\forall r\in ]0,1],\quad \lambda^{(n)}(B_r(x)\cap\Omega)\ge c\, r^n,
\end{equation*}
where $\lambda^{(n)}$ denotes Lebesgue's measure on $\R^n$.
In particular, domains with an outgoing cusp cannot be extension domains, see Figure~\ref{Fig:Domains}.
On the other hand, all $(\varepsilon,\delta)$-domains (i.e., locally uniform domains) are extension domains~\cite[Theorem 1]{jones_quasiconformal_1981}.

\begin{figure}[ht]
\centering
\hspace*{\fill}\includegraphics[height=5 cm]{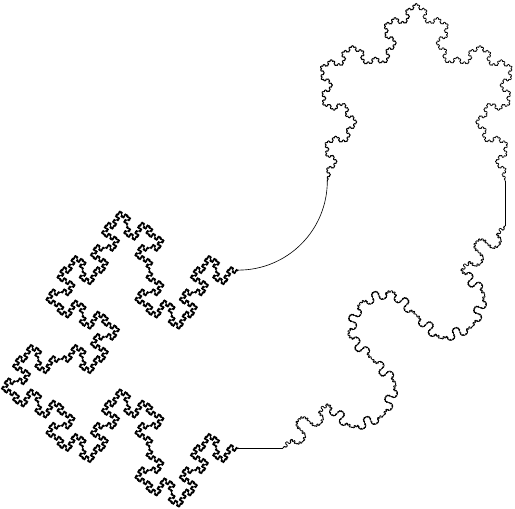}
\hfill\hfill\hfill
\includegraphics[height=5cm]{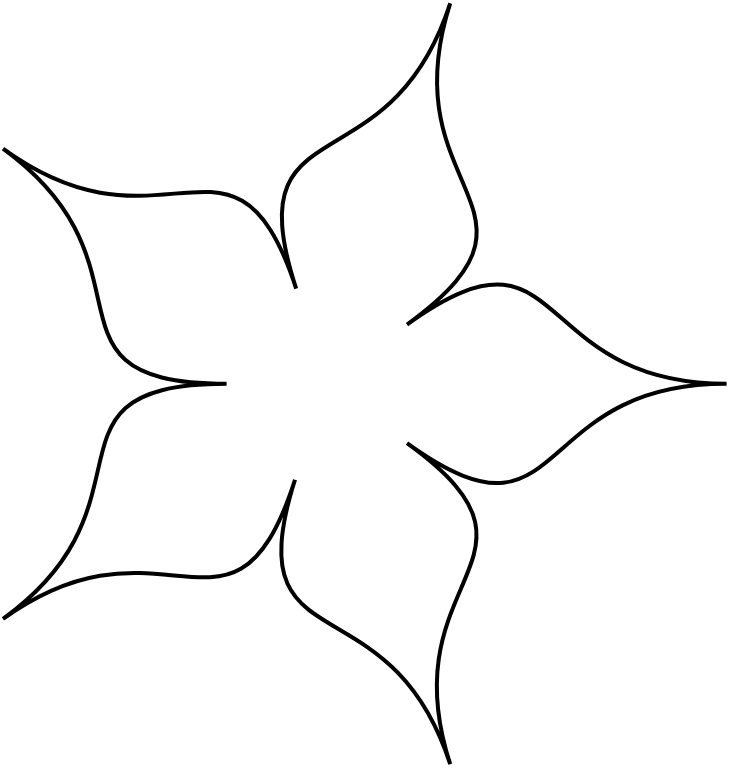}\hspace*{\fill}
\caption{On the left, an example of a two-sided admissible domain of $\R^2$ (it is locally uniform). Its boundary is composed of parts of different Hausdorff dimensions: three smooth lines of dimension $1$, a Koch curve of dimension $\frac{\ln 4}{\ln 3}$, a Minkowski curve of dimension $\frac{\ln 8}{\ln 4}$ and a pentagon-inspired fractal curve of dimension $\frac{\ln 14}{\ln 5}$.
On the right, a domain of $\R^2$ which is not an extension domain due to the presence of cusps.
The presence of ingoing cusps also prevents the complement domain from being an extension domain.}
\label{Fig:Domains}
\end{figure}

If $w\in H^1(\R^n)$, let $\tilde w$ denote one of its quasi-continuous representatives~\cite[Theorem 2.3.4]{chen_symmetric_2012}.
If $\Omega$ is an admissible domain, then for $u\in H^1(\Omega)$, the q.e. equivalence class $[(\mathrm{Ext}_\Omega u)^\sim|_{\bord}]$ of the restriction of $(\mathrm{Ext}_\Omega u)^\sim$ to $\bord$ depends on $u$ rather than on the choice of the extension $\mathrm{Ext}_\Omega$ or of the representative $(\mathrm{Ext}_\Omega u)^\sim$~\cite[Proposition 2.2]{claret_layer_2026}.
Denote by $\Hp$ the space of all $[\tilde w|_{\bord}]$ for $w\in H^1(\R^n)$.
Following~\cite[Theorem 6.1, Remark 6.2]{biegert_traces_2009}, we define the trace operator as in~\cite{claret_layer_2026}.

\begin{definition}[Trace operators]\label{Def:Trace}
Let $\Omega$ be a domain of $\R^n$. If $\Omega$ is admissible, then we define the interior trace operator on $\bord$ as
\begin{align*}
\Tr^i_{\bord}:H^1(\Omega)&\longrightarrow\Hp\\
u&\longmapsto[(\mathrm{Ext}_\Omega u)^\sim|_{\bord}],
\end{align*}
where $(\mathrm{Ext}_\Omega u)^\sim$ is a quasi-continuous representative of $\mathrm{Ext}_\Omega u\in H^1(\R^n)$.\\
If $\overline{\Omega}^c$ is admissible, then we define the exterior trace operator on $\bord$ as
\begin{align*}
\Tr^e_{\bord}:H^1(\overline{\Omega}^c)&\longrightarrow\Hp\\
u&\longmapsto[(\mathrm{Ext}_{\overline{\Omega}^c} u)^\sim|_{\bord}].
\end{align*}
If $\Omega$ is two-sided admissible, then we define the jump in trace across $\bord$ as
\begin{equation*}
\llbracket\Tr_{\bord}u\rrbracket\coloneq\Tr^i_{\bord}u-\Tr^e_{\bord}u,\quad u\in H^1(\nbord).
\end{equation*}
\end{definition}

In the definition of the jump in trace, we denoted $\Tr^i_{\bord}u$ (resp. $\Tr^e_{\bord}u$) for $u\in H^1(\nbord)$, in the sense of $\Tr^i_{\bord}(u|_\Omega)$ (resp. $\Tr^e_{\bord}(u|_{\overline{\Omega}^c})$).
When it is clear from the context, we shall continue to use such notations.

If $\Omega$ is an admissible domain, consider the space of $1$-harmonic functions in $H^1$
\begin{equation}\label{Eq:space-Vi}
V_{\mathrm i}(\Omega)\coloneq\big\{v\in H^1(\Omega)\;\big|\;\Delta u=u\mbox{ weakly on }\Omega\big\}.
\end{equation}
It is a closed subspace of $H^1(\Omega)$, and will be endowed with the norm $\|\cdot\|_{H^1(\Omega)}$.
It is also is closed in the Hilbert space
\begin{equation}\label{Eq:H1-Delta}
H^1_\Delta(\Omega)\coloneq\big\{u\in H^1(\Omega)\;\big|\;\Delta u\in L^2(\Omega)\big\},
\end{equation}
endowed with the norm
\begin{equation*}
\|\cdot\|_{H^1_{\Delta}(\Omega)}\coloneq\sqrt{\|\cdot\|_{H^1(\Omega)}^2+\|\Delta\cdot\|_{L^2(\Omega)}^2}.
\end{equation*}
It was proved in~\cite[Theorem 2.9]{claret_layer_2026} that the space $V_{\mathrm i}$ (denoted by $V_1$ in that other paper) is closely connected to the trace operator defined in Definition~\ref{Def:Trace}.
The following theorem relies notably on~\cite[Theorem 1]{hinz_existence_2021},~\cite[Theorem 5.1]{hinz_non-lipschitz_2021},~\cite[Theorem 2]{dekkers_mixed_2022} and~\cite[Corollary 2.3.1 and Example 2.3.1]{fukushima_dirichlet_2010} to establish fundamental properties of the trace operator.

\begin{theorem}[Trace theorem]\label{Th:Trace}
Let $\Omega$ be an admissible domain of $\R^n$. The following assertions hold.
\begin{enumerate}
\item[(i)] The space $H^1_0(\Omega)$, defined as the closure of $\mathscr{D}(\Omega)$ in $H^1(\Omega)$, is the kernel of $\Tr^i_{\bord}$ and the orthogonal complement of $V_{\mathrm i}(\Omega)$.
\item[(ii)] The space $\Hp$ is a Hilbert space when endowed with the norm
\begin{equation*}
\|f\|_{\Tr^i_{\bord}}\coloneq\min\big\{\|v\|_{H^1(\Omega)}\;\big|\;v\in H^1(\Omega)\mbox{ and }\Tr^i_{\bord}v=f\big\},
\end{equation*}
where the minimum is uniquely achieved by $v\in V_{\mathrm i}(\Omega)$ with $\Tr^i_{\bord}v=f$.
\item[(iii)] With respect to $\|\cdot\|_{\Tr^i_{\bord}}$, $\Tr^i_{\bord}$ is bounded with operator norm one.
Its restriction to $V_{\mathrm i}(\Omega)$ is an isometry onto $\Hp$.
\end{enumerate}
If $\overline{\Omega}^c$ is admissible, then counterparts of (i), (ii) and (iii) hold for $\Tr^e_{\bord}$ instead of $\Tr^i_{\bord}$, replacing $\Omega$ with $\overline{\Omega}^c$.
In particular, the norm
\begin{equation*}
\|f\|_{\Tr^e_{\bord}}\coloneq\min\big\{\|v\|_{H^1(\overline{\Omega}^c)}\;\big|\;v\in H^1(\overline{\Omega}^c)\mbox{ and }\Tr^e_{\bord}v=f\big\},
\end{equation*}
is a Hilbert space norm on $\Hp$.
\end{theorem}
When there is no ambiguity, we will simply denote $\Tr^i_{\bord}u=\Tr^iu$, $\Tr^e_{\bord}u=\Tr^eu$ and $\llbracket\Tr_{\bord}u\rrbracket=\llbracket\Tr u\rrbracket$. If $\Tr^iu=\Tr^eu$ (for instance, if $u\in H^1(\R^n)$), we denote the common value by $\Tr u$.

If $\Omega$ is a two-sided admissible domain, the norms $\|\cdot\|_{\Tr^i_{\bord}}$ and $\|\cdot\|_{\Tr^e_{\bord}}$ are equivalent~\cite[Proposition 3.4]{claret_convergence_2025}.
In what follows, we endow $\Hp$ with the norm $\|\cdot\|_{\Tr^i_{\bord}}$, which will be denoted by $\|\cdot\|_{\Hp}$, so that
\begin{equation*}
\|f\|_{\Hp}=\min\big\{\|v\|_{H^1(\Omega)}\;\big|\;v\in H^1(\Omega)\mbox{ and }\Tr^i_{\bord}v=f\big\}.
\end{equation*}
By Theorem~\ref{Th:Trace} (iii), the inner product on $\Hp$ with respect to the norm $\|\cdot\|_{\Hp}$ is charaterised by
\begin{equation}\label{Eq:B_Product}
\forall f,\varphi\in\Hp,\quad\langle f,\varphi\rangle_{\Hp}=\langle u_f,u_\varphi\rangle_{H^1(\Omega)},
\end{equation}
where $u_f,u_\varphi\in V_{\mathrm i}(\Omega)$ with $\Tr^iu_f=f$ and $\Tr^iu_\varphi=\varphi$.
In what follows, the dual space of $\Hp$ will be denoted by $\Hm\coloneq\mathcal{L}(\Hp,\C)$, and the associated dual pairing will be denoted by
\begin{equation*}
\langle g,f\rangle_{\Bcal'\!,\,\Hp},\quad g\in\Hm,\;f\in\Hp.
\end{equation*}

\begin{remark}
Given their definitions, the trace operators are local in the sense that if $u_1=u_2$ on a neighbourhood of the boundary, then their traces coincide.
For that matter, the definition of the trace operators can be extended to elements of $H^1_{\mathrm{loc_b}}(\Omega)$ (defined in the introduction): if $u\in H^1_{\mathrm{loc_b}}(\Omega)$, we define $\Tr^i_{\bord}u\coloneq\Tr^i_{\bord}(\chi u)$ where $\chi\in\mathscr D(\R^n)$ is such that $\chi=1$ on a neighbourhood of the boundary.
Note that the choice of $\chi$ is irrelevant. In the same way, we extend the exterior trace operator as $\Tr^e_{\bord}:H^1_{\mathrm{loc_b}}(\overline{\Omega}^c)\to\Hp$, and the jump in trace as $\llbracket\Tr_{\bord}\cdot\rrbracket:H^1_{\mathrm{loc_b}}(\nbord)\to\Hp$.
\end{remark}

The class of admissible domains is broader than that of Lipschitz domains.
It contains non-smooth domains which can have fractal or multifractal boundaries and boundaries of varying Hausdorff dimensions for instance, which may not be smooth anywhere.
For that matter, the classical definition of the normal derivative cannot be used, the normal vector field not being defined in general.
Instead, we use Green's formula to define a weak notion of normal derivative, similar to~\cite{lancia_transmission_2002,claret_layer_2026} for instance.
With that in mind and since the solutions to the Helmholtz equation are not expected to be integrable at infinity, we consider the local variant of the space $H^1_\Delta(\Omega)$ from~\eqref{Eq:H1-Delta}, namely
\begin{equation*}
H^1_{\Delta, \mathrm{loc_b}}(\Omega)\coloneq\big\{u\in L^2_{\mathrm{loc_b}}(\Omega)\;\big|\;\forall r>0,\;u|_{\Omega\cap B_r(0)}\in H^1_\Delta(\Omega\cap B_r(0))\big\}.
\end{equation*}

If $\Omega$ (resp. $\overline{\Omega}^c$) is an admissible domain of $\R^n$, denote by $\mathrm E_{\bord}:\Hp\to H^1(\Omega)$ (resp. $\Hp\to H^1(\overline{\Omega}^c)$) a continuous right inverse to the interior (resp. exterior) trace operator such that $(\mathrm E_{\bord}f)_{f\in\Hp}$ is uniformly compactly supported on $\R^n$.
For instance, one can define $\mathrm E_{\bord}$ as $(\chi\mathrm{H}_{\bord})|_\Omega$ (resp. $(\chi\mathrm{H}_{\bord})|_{\overline{\Omega}^c}$), where $\chi\in\mathscr D(\R^n)$ is such that $\chi=1$ near $\bord$ and $\mathrm{H}_{\bord}$ is the bilateral $1$-harmonic extension mapping $f\in\Hp$ to the unique $u\in H^1(\nbord)$ such that $\Delta u=u$ on $\nbord$.

\begin{definition}[Weak normal derivatives]\label{Def:ddn_loc}
Let $\Omega$ be a domain of $\R^n$. If $\Omega$ is admissible, then for all $u\in H^1_{\Delta, \mathrm{loc_b}}(\Omega)$, define its (weak) interior normal derivative on $\bord$ as
\begin{align*}
\ddny ui\Big|_{\bord}:\Hp &\longrightarrow \C\\
f &\longmapsto \int_{\Omega}(\Delta u)\overline{\mathrm E_{\bord}f}\,\dx+\int_{\Omega}\nabla u\cdot\nabla(\overline{\mathrm E_{\bord}f})\,\dx.
\end{align*}
If $\overline{\Omega}^c$ is admissible, then for all $u\in H^1_{\Delta, \mathrm{loc_b}}(\overline{\Omega}^c)$, define its (weak) exterior normal derivative on $\bord$ as
\begin{align*}
\ddny ue\Big|_{\bord}:\Hp &\longrightarrow \C\\
f &\longmapsto -\int_{\overline{\Omega}^c}(\Delta u)\overline{\mathrm E_{\bord}f}\,\dx-\int_{\overline{\Omega}^c}\nabla u\cdot\nabla(\overline{\mathrm E_{\bord}f})\,\dx.
\end{align*}
If $\Omega$ is two-sided admissible, then for all $u\in H^1_{\Delta, \mathrm{loc_b}}(\nbord)$, define its jump in normal derivative across $\bord$ as
\begin{equation*}
\left\llbracket\ddn u\Big|_{\bord}\right\rrbracket=\ddny ui\Big|_{\bord}-\ddny ue\Big|_{\bord}.
\end{equation*}
\end{definition}

The sign convention adopted in the definitions of the normal derivatives corresponds, in the classical case, to defining them both by means of the outgoing normal vector to $\Omega$.
When there is no ambiguity, we will simply denote $\ddny ui\big|_{\bord}=\ddny ui$, $\ddny ue\big|_{\bord}=\ddny ue$ and $\big\llbracket\ddn u\big|_{\bord}\big\rrbracket=\big\llbracket\ddn u\big\rrbracket$.
If $\ddny ui=\ddny ue$ (for instance, if $u\in H^1_{\Delta, \mathrm{loc_b}}(\R^n)$), we denote the common value by $\ddn u$.
We prove the normal derivatives defined thus are elements of $\Hm$ and satisfy Green's formula.
This implies in particular that Definition~\ref{Def:ddn_loc} does not depend on the choice of $\mathrm E_{\bord}$.

\begin{theorem}[Green's formula]\label{Th:Green}
Let $\Omega$ be an admissible domain of $\R^n$. Let $u\in H^1_{\Delta, \mathrm{loc_b}}(\Omega)$. The following assertions hold.
\begin{enumerate}
\item[(i)] The normal derivative $\ddny ui$ is an element of $\Hm$.
\item[(ii)] For all $v\in H^1_c(\Omega)\coloneq\{w\in H^1(\Omega)\;|\;\operatorname{supp}w\mathrm{\ is\ bounded}\}$, it holds
\begin{equation}\label{Eq:Green-ext}
\left\langle\ddny ui,\Tr^iv\right\rangle_{\Bcal'\!,\,\Bcal(\bord)}=\int_{\Omega}(\Delta u)\bar v\,\dx+\int_{\Omega}\nabla u\cdot\nabla\bar v\,\dx.
\end{equation}
\item[(iii)] If $u\in H^1_\Delta(\Omega)$, then~\eqref{Eq:Green-ext} holds for all $v\in H^1(\Omega)$. It follows that the operator $\ddny{}i$ maps $H^1_\Delta(\Omega)$ to $\Hm$ continuously, where $\Hm$ is endowed with $\|\cdot\|_{\Hm}$ which is subordinate to $\|\cdot\|_{\Hp}=\|\cdot\|_{\Tr_{\bord}^i}$.
\end{enumerate}
If $\overline{\Omega}^c$ is admissible, then counterparts of (i), (ii) and (iii) hold for $\ddny{}e$ instead of $\ddny{}i$, replacing $\Omega$ with $\overline{\Omega}^c$ and reversing the sign of the right-hand side of~\eqref{Eq:Green-ext}.
\end{theorem}

\begin{proof}
(i) By definition of $\mathrm E_{\bord}$, there exists a compact set $K$ such that for all $f\in\Hp$, $\operatorname{supp}(\mathrm{E}_{\bord}f)\subset K$. The result follows by the Cauchy-Schwarz inequality.\\
(ii) Let $\varphi\in \mathscr D(\Omega)$. Then, by weak derivation, it holds
\begin{equation}\label{Eq:Green-proof}
\int_{\Omega}(\Delta u)\bar \varphi\,\dx+\int_{\Omega}\nabla u\cdot\nabla\bar \varphi\,\dx=0.
\end{equation}
Since any element of $H^1_0(\Omega)\cap H^1_c(\Omega)$ is the limit of a sequence of $\mathscr D(\Omega)$ by definition, we can deduce that~\eqref{Eq:Green-proof} holds for all $\varphi\in H^1_0(\Omega)\cap H^1_c(\Omega)$.
By linearity, it follows that the right-hand side of~\eqref{Eq:Green-ext} only depends on $\Tr^iv$.\\
(iii) Let $(\chi_\ell)_{\ell\in\N}\in \mathscr D(\Omega)^\N$ be a non-decreasing sequence growing to $1$. Then, for all $v\in H^1(\Omega)$ and $\ell\in\N$, it holds
\begin{align*}
\left\langle\ddny ui,\Tr^iv\right\rangle_{\Bcal'\!,\,\Bcal(\bord)}&=\left\langle\ddny ui,\Tr^i(\chi_\ell v)\right\rangle_{\Bcal'\!,\,\Bcal(\bord)}\\
&=\int_{\Omega}(\Delta u)\chi_\ell\bar v\,\dx+\int_{\Omega}\nabla u\cdot\nabla(\chi_\ell\bar v)\,\dx.
\end{align*}
The identity follows by dominated convergence.
The continuity of $\ddny{}i$ follows from that of $\mathrm E_{\bord}$ and the Cauchy-Schwarz inequality.
\end{proof}

In particular -- and as expected --, the weak normal derivative only depends on the behaviour near the boundary.

\begin{lemma}
Let $\Omega$ be an admissible domain of $\R^n$. Let $u_1,u_2\in H^1_{\Delta, \mathrm{loc_b}}(\Omega)$. Assume $u_1=u_2$ on a neighbourhood of $\bord$. Then, it holds
\begin{equation*}
\ddny {u_1}i=\ddny{u_2}i.
\end{equation*}
If $\overline{\Omega}^c$ is admissible, then a counterpart of this statement holds for $\ddny{}e$ instead of $\ddny{}i$, replacing $\Omega$ with $\overline{\Omega}^c$.
\end{lemma}

\begin{proof}
By linearity, assume $u=0$ on a neighbourhood $U$ of $\bord$.
Let $\chi\in \mathscr D(\R^n)$ be supported on $U$ and equal to $1$ on a neighbourhood of $\bord$.
Since, for all $v\in H^1(\Omega)$, it holds $\Tr^iv=\Tr^i(\chi v)$, and $u$ and $\chi v$ have disjoint supports, the result follows from~\eqref{Eq:Green-ext}.
\end{proof}

The isometric properties of the interior trace operator (Theorem~\ref{Th:Trace} (iii)) are inherited by the interior normal derivation operator, and allow to express the Riesz isometry on the trace space as the Poincaré-Steklov (Dirichlet-to-Neumann) operator for $(\Delta-1)$ on the interior domain~\cite[Lemma 3.1, Eq.~(21)]{claret_layer_2026}.

\begin{proposition}[{\cite{claret_layer_2026}}]
Let $\Omega$ be an admissible domain.
The restriction to $V_{\mathrm i}(\Omega)$ of the interior normal derivation operator $\ddny{}i$ is isometric onto $\Hm$.
The Riesz isometry between the trace space $\Hp$ and its dual $\Hm$ is given by the Poincaré-Steklov operator for $(\Delta-1)$ on $\Omega$, that is
\begin{equation}\label{Eq:PS-i}
\begin{aligned}
\Lambda_{\mathrm i}: \Hp & \longrightarrow \Hm\\
\Tr^i u&\longmapsto \ddny ui,
\end{aligned}
\end{equation}
where $u\in V_{\mathrm i}(\Omega)$ defined by~\eqref{Eq:space-Vi}.
\end{proposition}
In particular, it follows from~\eqref{Eq:B_Product} that the dual pairing $(\Hm,\Hp)$ is characterised, for $f\in\Hp$ and $g\in\Hm$, by
\begin{equation*}
\langle g,f\rangle_{\Bcal'\!,\,\Hp}=\langle\Lambda_{\mathrm i}^{-1}g,f\rangle_{\Hp}=\langle u,v\rangle_{H^1(\Omega)},
\end{equation*}
where $u,v\in V_{\mathrm i}(\Omega)$ with $\ddny ui=g$ and $\Tr^iv=f$.

\section{Helmholtz transmission problem}\label{Sec:Tr-Prob}

The trace and normal derivation operators introduced in Section~\ref{Sec:Framework} allow to use variational methods to solve problems set on spaces of locally $H^1$ functions (in the sense of bounded sets specified in the introduction).
If $\Omega$ is a two-sided admissible domain of $\R^n$ and if $k>0$, consider the following transmission problem for the Helmholtz equation
\begin{equation}\label{Eq:Helm-Prob}
\begin{cases}
(\Delta+k^2)u=0 &\mbox{on }\nbord,\\
\llbracket\Tr_{\bord} u\rrbracket=f,\\
\llbracket\ddn u|_{\bord}\rrbracket=g,\\
\mathrm{SRC},
\end{cases}
\end{equation}
for given jumps in trace $f\in\Hp$ and in normal derivative $g\in\Hm$, where $\mathrm{SRC}$ stands for the Sommerfeld Radiation Condition~\cite{schot_eighty_1992, sommerfeld_greensche_1912}, that is
\begin{equation}\label{Eq:SRC}
\left(\frac\partial{\partial|x|}-\mathrm ik\right)u(x)=\underset{|x|\to\infty}{o}\big(|x|^{\frac{1-n}2}\big)\qquad\mbox{uniformly in }\frac x{|x|}.
\end{equation}
That condition models the fact that the energy of the underlying wave propagates towards infinity, instead of being infused to the system from infinity (which would be unphysical).

\subsection{Layer potential operators}\label{Subsec:LP}

The classical approach to solve the transmission problem~\eqref{Eq:Helm-Prob} is to decompose the solution into parts with a homogeneous jump in trace or in normal derivative.
Those parts are then determined by means of the layer potential operators~\cite{vladimirov_equations_1971, verchota_layer_1984}.

\subsubsection{Single layer potential}\label{Subsubsec:SLP}

First, let us consider~\eqref{Eq:Helm-Prob} with a homogeneous jump in trace, that is, for $g\in\Hm$,
\begin{equation}\label{Eq:SLP-Helm-Prob}
\begin{cases}
(\Delta+k^2)u=0 &\mbox{on }\nbord,\\
\llbracket\Tr_{\bord} u\rrbracket=0,\\
\llbracket\ddn u|_{\bord}\rrbracket=g,\\
\mathrm{SRC}.
\end{cases}
\end{equation}
The solutions to~\eqref{Eq:SLP-Helm-Prob} are not expected to be $H^1$ at infinity.
To derive a variational formulation for the Helmholtz transmission problem, we introduce a truncation (see~\cite[§~2.6.4]{nedelec_acoustic_2011}).
Let $B$ be a large ball containing $\Omega$.
If $u\in H^1_{\mathrm{loc_b}}(\R^n)$ is a solution to~\eqref{Eq:SLP-Helm-Prob}, then, for all $v\in H^1(B)$, it holds
\begin{equation}\label{Eq:SLP-Helm-FV}
\int_{B}\nabla u\cdot\nabla\bar v\,\dx-k^2\int_{B}u\bar v\,\dx-\int_{\partial B}\tilde\Lambda_{\partial B}(\Tr_{\partial B}u)\Tr_{\partial B}\bar v\,\dsig=\langle g,\Tr_{\bord}v\rangle_{\Bcal'\!,\,\Bcal(\bord)},
\end{equation}
where $\tilde\Lambda_{\partial B}$ denotes the Poincaré-Steklov operator for the Helmholtz problem outside $B$ with the $\mathrm{SRC}$, defined by
\begin{equation}\label{Eq:Lambda_tilde}
\begin{aligned}
\tilde\Lambda_{\partial B}:\Bcal(\partial B)&\longrightarrow\Bcal'(\partial B)\\
\Tr_{\partial B}^ev &\longmapsto \ddny ve\Big|_{\partial B},
\end{aligned}
\end{equation}
where $(\Delta+k^2)v=0$ on $\overline{B}^c$ and $v\in H^1_{\mathrm{loc_b}}(\overline{B}^c)$ satisfies the $\mathrm{SRC}$~\eqref{Eq:SRC}.

\begin{remark}\label{Rem:PS_B'-L2}
In~\eqref{Eq:SLP-Helm-FV}, $\tilde\Lambda_{\partial B}(\Tr_{\partial B}u)$ is thought of as an element of $L^2(\partial B,\sigma)$ instead of $\Bcal'(\partial B)$.
Since $B$ is smooth, then, up to equivalent norms, $\Bcal'(\partial B)$ coincides with $H^{-\frac12}(\partial B)$~\cite[Chapter IV, Appendix]{dautray_mathematical_1988-1}, so that $L^2(\partial B,\sigma)\hookrightarrow\Bcal'(\partial B)$.
Moreover, since $u$ is a solution to~\eqref{Eq:SLP-Helm-Prob}, it holds $\tilde\Lambda_{\partial B}(\Tr_{\partial B}u)=\ddn u|_{\partial B}\in C^\infty(\partial B)$ by elliptic regularity, which allows to write the boundary term on $\partial B$ as an integral.
\end{remark}

It is known~\cite[Lemma 2.1]{chandler-wilde_wave-number-explicit_2008} that for all $f\in\Bcal(\partial B)$, it holds
\begin{equation}\label{Eq:Ineq-Re-PS}
\mathrm{Re}\left(\int_{\partial B}(\tilde\Lambda_{\partial B} f)\bar f\,\dsig \right)\le 0.
\end{equation}

The fact that any solution to the transmission problem~\eqref{Eq:SLP-Helm-Prob} solves the variational formulation~\eqref{Eq:SLP-Helm-FV} follows from Green's formula (Theorem~\ref{Th:Green}).
The converse is also true.

\begin{lemma}\label{Lem:Truncation}
Let $\Omega$ be a two-sided admissible domain of $\R^n$ and let $k>0$.
Let $B$ be a large ball containing $\Omega$.
Then, the restriction to $B$ of any weak solution to~\eqref{Eq:SLP-Helm-Prob} solves~\eqref{Eq:SLP-Helm-FV}.
Conversely, any solution to~\eqref{Eq:SLP-Helm-FV} can be extended as a weak solution to~\eqref{Eq:SLP-Helm-Prob}.
\end{lemma}

\begin{proof}
Let $\mathrm{E}_{\partial B}:\Bcal(\partial B)\to H^1_{\mathrm{loc_b}}(\overline{B}^c)$ denote the exterior Dirichlet extension from $\partial B$ to $\overline{B}^c$ for the Helmholtz problem with $\mathrm{SRC}$, that is, for $f\in\Bcal(\partial B)$,
\begin{equation*}
\begin{cases}
(\Delta+k^2)(\mathrm{E}_{\partial B}f)=0 &\mbox{on }\overline{B}^c,\\
\Tr^e_{\partial B}(\mathrm{E}_{\partial B}f)=f,\\
\mathrm{SRC}.
\end{cases}
\end{equation*}
If $u\in H^1(B)$ is a weak solution to~\eqref{Eq:SLP-Helm-FV}, then it satisfies $\ddn u|_{\partial B}=\tilde\Lambda_{\partial B}(\Tr_{\partial B}u)$.
Denote $\tilde u\coloneq u\oplus\mathrm{E}_{\partial B}(\Tr_{\partial B}u)\in H^1_{\mathrm{loc_b}}(\R^n)$.
Then, $(\Delta+k^2)\tilde u=0$ on $\R^n\backslash(\partial B\cup\bord)$, $\llbracket\Tr_{\partial B}\tilde u\rrbracket=0$ and $\llbracket\ddn{\tilde u}|_{\partial B}\rrbracket=0$, so that $(\Delta+k^2)\tilde u=0$ on $\nbord$.
Therefore, $\tilde u$ is a weak solution to~\eqref{Eq:SLP-Helm-Prob}.
\end{proof}

The well-posedness of~\eqref{Eq:SLP-Helm-FV} -- hence, that of~\eqref{Eq:SLP-Helm-Prob} -- follows from Fredholm's alternative.

\begin{lemma}\label{Lem:Op-Fredholm}
Let $\Omega$ be a bounded two-sided $H^1$-admissible domain of $\R^n$ and let $k>0$.
Let $B$ be a large ball containing $\Omega$.
Then, the operator $A:H^1(B)\to H^1(B)$ associated to the left-hand side of~\eqref{Eq:SLP-Helm-FV}, namely, for $u,v\in H^1(B)$,
\begin{equation*}
\langle Au,v\rangle_{H^1(B)}\coloneq\int_{B}\nabla u\cdot\nabla\bar v\,\dx-k^2\int_{B}u\bar v\,\dx-\int_{\partial B}\tilde\Lambda_{\partial B}(\Tr_{\partial B}u)\Tr_{\partial B}\bar v\,\dsig,
\end{equation*}
is Fredholm (of index $0$, see~\cite[§~1.3]{colton_integral_2013}).
\end{lemma}

\begin{proof}
By~\eqref{Eq:Ineq-Re-PS}, it holds, for $u\in H^1(B)$,
\begin{equation}\label{Eq:SLP_FV_Coercive}
\int_{B}|\nabla u|^2\,\dx+\int_{B}|u|^2\,\dx-\mathrm{Re}\left(\int_{\partial B}\tilde\Lambda_{\partial B}(\Tr_{\partial B}u)\Tr_{\partial B}\bar u\,\dsig\right)
\ge \|u\|_{H^1(B)}^2.
\end{equation}
Since $B$ is bounded, $H^1(B)\hookrightarrow L^2(B)$ is compact, so that $A_1:H^1(B)\to H^1(B)$ defined by
\begin{equation*}
\langle A_1u,v\rangle_{H^1(B)}\coloneq-(k+1)\int_Bu\bar v\,\dx
\end{equation*}
can be seen as a compact perturbation of the coercive contribution~\eqref{Eq:SLP_FV_Coercive}.
\end{proof}

It follows that the transmission problem~\eqref{Eq:SLP-Helm-FV} is well-posed since the homogeneous problem is uniquely solved.

\begin{proposition}
Let $\Omega$ be a two-sided admissible domain of $\R^n$ and let $k>0$.
Then, the transmission problem~\eqref{Eq:SLP-Helm-Prob} is well-posed on $H^1_{\mathrm{loc_b}}(\R^n)$.
\end{proposition}

\begin{proof}
Let $u\in H^1(B)$ be a solution to~\eqref{Eq:SLP-Helm-FV} with $g=0$.
Then, by Lemma~\ref{Lem:Truncation}, $\tilde u\coloneq u\oplus\mathrm{E}_{\partial B}(\Tr_{\partial B}u)\in H^1_{\mathrm{loc_b}}(\R^n)$ satisfies $(\Delta+k^2)\tilde u=0$ on $\R^n$ and the $\mathrm{SRC}$: $\tilde u=0$ (hence, $u=0$). By Lemma~\ref{Lem:Op-Fredholm}, it follows that~\eqref{Eq:SLP-Helm-FV} -- hence,~\eqref{Eq:SLP-Helm-Prob} -- is well-posed.
\end{proof}

The well-posedness of~\eqref{Eq:SLP-Helm-Prob} allows to define the associated single layer potential operator as the operator reconstructing the solution given the jump in normal derivative.

\begin{definition}\label{Def:SLP-Helm}
Let $\Omega$ be a two-sided admissible domain of $\R^n$ and let $k>0$.
The single layer potential operator associated to the Helmholtz transmission problem~\eqref{Eq:SLP-Helm-Prob} is defined as
\begin{align*}
\Scal_k:\Bcal'(\bord)&\longrightarrow H^1_{\mathrm{loc_b}}(\R^n),\\
g&\longmapsto u,
\end{align*}
where $u$ is the unique weak solution to~\eqref{Eq:SLP-Helm-Prob} associated to $g$.
\end{definition}

By construction, the single layer potential operator is both linear and continuous.

\begin{proposition}
Let $\Omega$ be a two-sided admissible domain of $\R^n$ and let $k>0$.
The single layer potential $\Scal_k:\Bcal'(\bord)\to H^1_{\mathrm{loc_b}}(\R^n)$ is linear and continuous.
\end{proposition}

\begin{proof}
The linearity of the operator is inherited from that of the problem.
The continuity follows from the continuity of $\mathrm E_{\partial B}:\Bcal(\partial B)\to H^1_{\mathrm{loc_b}}(\overline{B}^c)$ (see for instance~\cite[§~2.6.5]{nedelec_acoustic_2011}) and Fredholm's alternative.
\end{proof}

\subsubsection{Double layer potential}\label{Subsubsec:DLP}

We now turn to~\eqref{Eq:Helm-Prob} with a homogeneous jump in normal derivative, that is, for $f\in\Hp$,
\begin{equation}\label{Eq:DLP-Helm-Prob}
\begin{cases}
(\Delta+k^2)u=0 &\mbox{on }\nbord,\\
\llbracket\Tr_{\bord} u\rrbracket=-f,\\
\llbracket\ddn u|_{\bord}\rrbracket=0,\\
\mathrm{SRC}.
\end{cases}
\end{equation}
To derive the variational formulation associated to~\eqref{Eq:DLP-Helm-Prob}, we lift the Dirichlet boundary condition.
To that end, consider the $1$-harmonic double layer potential operator applied to $f$, $\Dcal_{\mathrm i}f\in H^1(\nbord)$~\cite[§~4.3]{claret_layer_2026}, characterised by
\begin{equation}\label{Eq:DLP-1-harmo-Prob}
\begin{cases}
(-\Delta+1)\Dcal_{\mathrm i}f=0 &\mbox{on }\nbord,\\
\llbracket\Tr_{\bord}\Dcal_{\mathrm i}f\rrbracket=-f,\\
\llbracket\ddn{}\Dcal_{\mathrm i}f|_{\bord}\rrbracket=0.
\end{cases}
\end{equation}
Note that $\Dcal_{\mathrm i}:\Hp\to H^1(\nbord)$ is continuous.
Let $\chi\in \mathscr D(\R^n)$ be such that $\chi|_{B}=1$.
Setting $w=u-\chi(\Dcal_{\mathrm i}f)$, the transmission problem can be lifted as
\begin{equation}\label{Eq:DLP-Helm-Lift}
\begin{cases}
(\Delta+k^2)w=-(\Delta+k^2)(\chi(\Dcal_{\mathrm i}f)) &\mbox{on }\nbord,\\
\llbracket\Tr_{\bord} w\rrbracket=0,\\
\llbracket\ddn w|_{\bord}\rrbracket=0,\\
\mathrm{SRC}.
\end{cases}
\end{equation}

Using the same truncation method as in the previous subsection, a solution $w$ to~\eqref{Eq:DLP-Helm-Lift} satisfies, for all $v\in H^1(B)$,
\begin{equation}\label{Eq:DLP-Helm-FV}
\int_{B}\nabla w\cdot\nabla\bar v\,\dx-k^2\int_{B}w\bar v\,\dx-\int_{\partial B}\tilde\Lambda_{\partial B}(\Tr_{\partial B}w)\Tr_{\partial B}\bar v\,\dsig=(k^2+1)\int_{B}(\Dcal_{\mathrm i}f)\bar v\,\dx,
\end{equation}
and the same considerations allow to prove the problem is well-posed.
Note that, extending the solution on the truncated domain to the whole space in the spirit of Lemma~\ref{Lem:Truncation}, we will use an extension similar to $\mathrm{E}_{\partial B}$ but accounting for the (smooth) right-hand side in~\eqref{Eq:DLP-Helm-Lift}.
Such an extension is still continuous by the well-posedness of the Helmholtz problem with Dirichlet condition and $\mathrm{SRC}$ outside a sphere~\cite{colton_integral_2013}.

\begin{proposition}
Let $\Omega$ be a two-sided admissible domain of $\R^n$ and let $k>0$.
Then, the transmission problem~\eqref{Eq:DLP-Helm-Prob} is well-posed on $H^1_{\mathrm{loc_b}}(\nbord)$.
\end{proposition}

\begin{remark}
Alternatively, the well-posedness of the transmission problem~\eqref{Eq:DLP-Helm-Prob} could have been obtained without the truncation method: $(\Delta+k^2)(\chi(\Dcal_{\mathrm i}f))\in L^2(\R^n)$ along with~\eqref{Eq:DLP-Helm-Lift} allow to restate the lifted problem as
\begin{equation*}
\begin{cases}
(\Delta+k^2)w=-(\Delta+k^2)(\chi(\Dcal_{\mathrm i}f)) &\mbox{on }\R^n,\\
\mathrm{\mathrm{SRC}},
\end{cases}
\end{equation*}
which is well-posed (see Subsubsection~\ref{Subsec:LP-Repr}).
\end{remark}

As in the previous subsection, the well-posedness of~\eqref{Eq:DLP-Helm-Prob} allows to define the associated double layer potential operator as the operator reconstructing the solution from the jump in trace.

\begin{definition}\label{Def:DLP-Helm}
Let $\Omega$ be a two-sided admissible domain of $\R^n$ and let $k>0$.
The double layer potential operator associated to the Helmholtz transmission problem~\eqref{Eq:DLP-Helm-Prob} is defined as
\begin{align*}
\Dcal_k:\Bcal(\bord)&\longrightarrow H^1_{\mathrm{loc_b}}(\nbord),\\
f&\longmapsto u,
\end{align*}
where $u$ is the unique weak solution to~\eqref{Eq:DLP-Helm-Prob} associated to $f$.
\end{definition}

Once again, it follows from its construction that the double layer potential operator is linear and continuous.

\begin{proposition}
Let $\Omega$ be a two-sided admissible domain of $\R^n$ and let $k>0$.
The double layer potential $\Dcal_k:\Bcal'(\bord)\to H^1_{\mathrm{loc_b}}(\nbord)$ is linear and continuous.
\end{proposition}

\begin{proof}
The lifted problem is well-posed and $\Dcal_{\mathrm i}:\Hp\to H^1(\nbord)$ is linear and continuous~\cite[Lemma 4.8]{claret_layer_2026}.
\end{proof}


By construction, the layer potentials allow to recover the solution to~\eqref{Eq:Helm-Prob} given the jumps $f\in\Hp$ and $g\in\Hm$, thus generalising Green's third identity in this measure-free framework.

\begin{theorem}[Green's third identity]\label{Th:Green3}
Let $\Omega$ be a two-sided admissible domain of $\R^n$ and let $k>0$.
Given $f\in\Hp$ and $g\in\Hm$, the Helmholtz transmission problem~\eqref{Eq:Helm-Prob} has a unique solution $u\in H^1_{\mathrm{loc_b}}(\nbord)$ given by
\begin{equation*}
u=\Scal_kg-\Dcal_kf.
\end{equation*}
\end{theorem}

\begin{remark}\label{Rem:Helm_Ext_WP}
If $\Omega$ is a domain with compact boundary such that $\overline{\Omega}^c$ is $H^1$-admissible, similar arguments to those given in this section yield the well-posedness on $H^1_{\locb}(\overline{\Omega}^c)$ of the exterior Dirichlet and Neumann problems with the $\mathrm{\mathrm{SRC}}$~\eqref{Eq:SRC}, that is
\begin{equation*}
\begin{cases}
\Delta u+k^2u=0 &\mbox{on }\overline{\Omega}^c,\\
\Tr^e_{\bord}u=f,\\
\mathrm{\mathrm{SRC}},
\end{cases}
\qquad\mbox{and}\qquad
\begin{cases}
\Delta u+k^2u=0 &\mbox{on }\overline{\Omega}^c,\\
\ddny ue|_{\bord}=g,\\
\mathrm{\mathrm{SRC}},
\end{cases}
\end{equation*}
respectively, for $f\in\Hp$ and $g\in\Hm$.
\end{remark}

\subsubsection{Representation of the layer potentials}\label{Subsec:LP-Repr}

As in the Lipschitz case (for instance,~\cite{colton_integral_2013,mclean_strongly_2000}), we aim at representing the layer potential operators by means of the fundamental solution to the Helmholtz equation (which will be referred to as the Newtonian potential) $\mathcal G_k$, characterised by
\begin{equation*}
-(\Delta+k^2)\mathcal G_ku=-\mathcal G_k(\Delta+k^2)u=u,
\end{equation*}
for all compactly supported $u\in\mathscr D'(\R^n)$.
That potential is known to be associated to the so-called Newtonian kernel $G_k$, in the sense that for all compactly supported $u\in\mathscr D'(\R^n)$, it holds
\begin{equation*}
\mathcal G_ku=G_k*u.
\end{equation*}
The kernel $G_k\in C^\infty(\R^n\backslash\{0\})$ satisfies
\begin{equation*}
-(\Delta+k^2)G_k=\delta,
\end{equation*}
where $\delta\in\mathscr D'(\R^n)$ denotes the Dirac distribution.
For $k>0$, the Newtonian kernel $G_k$ can be expressed as (see for instance~\cite[Equation (9.14)]{mclean_strongly_2000})
\begin{equation}\label{Eq:Hankel}
G_k(z)=\frac {\mathrm i}4\left(\frac k{2\pi|z|}\right)^{\frac {n-2}2}H^{(1)}_{\frac{n-2}2}(k|z|),\qquad z\in \R^n\backslash\{0\},
\end{equation}
where $H^{(1)}_m$ denotes the Hankel function of the first kind and order $m\in\C$.

To represent the layer potential operators $\Scal_k$ and $\Dcal_k$ in terms of $\mathcal G_k$, we follow the method outlined in~\cite[Chapter 6]{mclean_strongly_2000}.
Before proceeding, let us recall well-known properties of the Newtonian potential regarding its regularity and behavior at infinity.
\begin{proposition}[{\cite[Theorem 3.1.2]{sauter_boundary_2011}}]\label{Prop:Newt-Pot-Reg}
For $s\in\R$, the mapping $\mathcal G_k:H^s_c(\R^n)\to H^{s+2}_{\mathrm{loc_b}}(\R^n)$ is continuous, where $H^s_c(\R^n)$ consists of all the compactly supported elements of $H^s(\R^n)$.
\end{proposition}

\begin{proposition}[{\cite[Proposition 7.14 and Theorem 9.6]{mclean_strongly_2000}}]\label{Prop:Newt-Pot-SRC}
If $u\in \mathscr D'(\R^n)$ is compactly supported, then $\mathcal G_ku=G_k*u$ satisfies the $\mathrm{SRC}$~\eqref{Eq:SRC}.
\end{proposition}

So far, the layer potentials have been understood as operators allowing to solve the Helmholtz transmission problem~\eqref{Eq:Helm-Prob} and reconstruct the solution by means of Green's third identity (Theorem~\ref{Th:Green3}).
To derive formulae for the layer potentials, we prove a similar representation formula involving the Newtonian potential.
Assume $u\in L^2(\R^n)$ is such that $u\in H^1_{\locb}(\nbord)$ and $-(\Delta+k^2)u=h$ weakly on $\nbord$ for some $h\in H^{-1}(\R^n)$.
Then, for all $\varphi\in \mathscr D(\R^n)$, it holds
\begin{align*}
\langle(-\Delta-k^2)u,\varphi\rangle_{\mathscr D'\!,\,\mathscr D(\R^n)}&=\int_{\R^n}u(-\Delta-k^2)\bar\varphi\,\dx\\
&=\int_{\R^n}\nabla u\cdot\nabla\bar\varphi\,\dx-k^2\int_{\R^n}u\bar\varphi\,\dx-\overline{\left\langle\ddn\varphi,\llbracket\Tr u\rrbracket\right\rangle}\\
&=\langle h,\varphi\rangle_{H^{-1}\!,\,H^1(\R^n)}+\left\langle\left\llbracket\ddn u\right\rrbracket,\Tr\varphi\right\rangle-\overline{\left\langle\ddn\varphi,\llbracket\Tr u\rrbracket\right\rangle},
\end{align*}
where the unspecified dual pairings are on $(\Hm,\Hp)$.
This can be rewritten as an identity in $\mathscr D'(\R^n)$ as
\begin{equation}\label{Eq:Helm-Repr}
-(\Delta+k^2)u=h+\Tr^*\left\llbracket\ddn u\right\rrbracket-\left(\ddn{}\right)^*\llbracket\Tr u\rrbracket.
\end{equation}
Moreover, if $u$ (and $h$) is compactly supported on $\R^n$, applying the Newtonian potential to the above yields
\begin{equation}\label{Eq:Pot-Repr}
u=\mathcal G_kh+\mathcal G_k\Tr^*\left\llbracket\ddn u\right\rrbracket-\mathcal G_k\left(\ddn{}\right)^*\llbracket\Tr u\rrbracket\qquad\mbox{in }\mathscr D'(\R^n),
\end{equation}
bearing in mind that the adjoint operators $\Tr^*:\Hm\to H^{-1}(\R^n)$ and $(\ddn{})^*:\Hp\to (H^1_\Delta(\R^n))'$ both have ranges composed of distributions (compactly) supported on $\bord$, see Section~\ref{Sec:Framework}.
The second term on the right-hand side can be identified as a single layer potential.

\begin{proposition}\label{Prop:SLP-Repr}
Let $\Omega$ be a two-sided admissible domain of $\R^n$ and let $k>0$.
For all $g\in\Hm$, it holds
\begin{equation*}
\Scal_k g=\mathcal G_k\Tr^*g.
\end{equation*}
\end{proposition}

\begin{proof}
For $g\in\Hm$, denote $u\coloneq\mathcal G_k\Tr^*g$.
By definition of $\mathcal G_k$, it holds
\begin{equation*}
-(\Delta+k^2)u=\Tr^*g\qquad\mbox{in }\mathscr D'(\R^n).
\end{equation*}
In particular, $(\Delta+k^2)u=0$ on $\nbord$.
Moreover, by Proposition~\ref{Prop:Newt-Pot-Reg}, it holds $u\in H^1_{\mathrm{loc_b}}(\R^n)$, hence $\llbracket\Tr_{\bord} u\rrbracket=0$.
If $\chi\in \mathscr D(\R^n)$ is such that $\chi=1$ on an open neighbourhood $U$ of $\overline{\Omega}$, then~\eqref{Eq:Pot-Repr} for $\chi u$ yields
\begin{equation*}
u=\mathcal G_k\Tr^*\left\llbracket\ddn u\right\rrbracket\qquad\mbox{in }\mathscr D'(U).
\end{equation*}
Since $\Tr(\mathscr D(U))=C^\infty(\bord)$ is dense in $\Hp$, it follows that $\llbracket\ddn u\rrbracket=g$.
By Proposition~\ref{Prop:Newt-Pot-SRC}, the $\mathrm{SRC}$ is satisfied by $u$, and the result follows.
\end{proof}

Now that the second term on the right-hand side of~\eqref{Eq:Pot-Repr} has been identified, we can deduce the last term corresponds to the double layer potential.

\begin{proposition}\label{Prop:DLP-Repr}
Let $\Omega$ be a two-sided admissible domain of $\R^n$ and let $k>0$.
For all $f\in\Hp$, it holds
\begin{equation*}
\Dcal_k f=\mathcal G_k\left(\ddn{}\right)^*f.
\end{equation*}
\end{proposition}

\begin{proof}
Let $\eta\in\R$ be such that $\mathrm{Re}(k^2)+\eta<0$. Then, for all $f\in\Hp$, the Dirichlet problem
\begin{equation*}
\begin{cases}
(\Delta+k^2+\eta)w=0 &\mbox{on }\Omega,\\
\Tr w=f,
\end{cases}
\end{equation*}
has a unique solution $w\in H^1(\Omega)$. In particular, for $f\in\mathscr D(\R^n)$, consider $u\coloneq w\mathds{1}_\Omega\in H^1(\nbord)$. Then, $(-\Delta-k^2)u=\eta u$ on $\nbord$, so that by~\eqref{Eq:Pot-Repr}, it holds
\begin{equation*}
u=\eta \mathcal G_ku+\Scal_k\left\llbracket\ddn u\right\rrbracket- \mathcal G_k\left(\ddn{}\right)^*f.
\end{equation*}
Since $u\in L^2(\R^n)$, then $\mathcal G_ku\in H^2_{\mathrm{loc_b}}(\R^n)$ (Proposition~\ref{Prop:Newt-Pot-Reg}) and therefore has no jump in trace nor in normal derivative across $\bord$.
Bearing in mind the jump relations satisfied by the single layer potential (see Definition~\ref{Def:SLP-Helm}), it holds
\begin{equation*}
\left\llbracket\Tr\left(\mathcal G_k\left(\ddn{}\right)^*f\right)\right\rrbracket=-f\quad\mbox{and}\quad\left\llbracket\ddn{}\left(\mathcal G_k\left(\ddn{}\right)^*f\right)\right\rrbracket=0.
\end{equation*}
Since $(\ddn{})^*f$ is compactly supported on $\bord$, it holds
\begin{equation*}
(\Delta+k^2)\left(\mathcal G_k\left(\ddn{}\right)^*f\right)=0\quad\mbox{on }\nbord.
\end{equation*}
It follows from Proposition~\ref{Prop:Newt-Pot-SRC} and Definition~\ref{Def:DLP-Helm} that $\Dcal_kf=\mathcal G_k(\ddn{})^*f$ for all $f\in\mathscr D(\R^n)$, and all $f\in\Hp$ by density.
\end{proof}

If the boundary is endowed with a `suitable' Borel measure $\mu$ allowing for irregular geometries (see~\cite[Remark 3.7]{chandler-wilde_integral_2026}, relying on~\cite{hinz_kernels_2025}), then it holds
\begin{equation*}
\Hp\hookrightarrow L^2(\bord,\mu)\hookrightarrow\Hm,
\end{equation*}
where the embeddings (i.e., linear, continuous injections) are dense and compact, and where $L^2(\bord,\mu)$ is identified with its dual space.
In that sense, an element $g\in\Hm$ is said to be an element of $L^2(\bord,\mu)$ when there exists $\tilde g\in L^2(\bord,\mu)$ such that
\begin{equation}\label{Eq:Incl-L2-B'}
\forall f\in\Hp,\quad\langle g,f\rangle_{\Bcal'\!,\,\Bcal(\bord)}=\langle\tilde g,f\rangle_{L^2(\bord,\mu)}.
\end{equation}
In that case, we identify $\tilde g$ and $g$, and simply denote $g\in L^2(\bord,\mu)$.
When $g\in L^2(\bord,\mu)$, it follows from Proposition~\ref{Prop:SLP-Repr} that its single layer potential can be expressed as a boundary integral: given $\chi\in\mathscr D(\R^n)$ such that $\chi=1$ on a neighbourhood of $\bord$, $\chi\Tr^*g=\Tr^*g$ so that, for a.e. $x\in\R^n$, it holds
\begin{align}\label{Eq:Link_Sk_Gk}
\Scal_kg(x)&=\langle \Tr^*g,\chi\,\overline{G_k(x-\cdot)}\rangle_{H^{-1}\!,\,H^1(\R^n)}\\
&=\langle g,\Tr\overline{G_k(x-\cdot)}\rangle_{\Bcal'\!,\,\Hp}=\int_{\bord}G_k(x-y)g(y)\,\mu(\dy),\nonumber
\end{align}
where the first equality follows from the definition of the embedding $H^{-1}(\R^n)\hookrightarrow \mathscr D'(\R^n)$ and the fact that $\Tr^*g$ is compactly supported on $\bord$, and the last one from the meaning of $L^2(\bord,\mu)\hookrightarrow\Hm$.
In particular, the first two equalities are valid even in the absence of $\mu$, for $g\in\Hm$.

Similarly, the double layer potential $\Dcal_k$ can be represented by means of the Newtonian kernel $G_k$ as
\begin{equation*}
\Dcal_kf(x)=\Big\langle\ddn{G_k(x-\cdot)},\bar f\Big\rangle_{\Bcal'\!,\,\Hp},\quad\mbox{a.e. }x\in\nbord.
\end{equation*}
To express that operator as an integral, unlike the single layer potential, the condition does not regard the boundary condition but the domain itself. 
Indeed, one must have $\ddn {G_k(x-\cdot)}\big|_{\bord}\in L^2(\bord,\mu)\hookrightarrow\Hm$ for $x\in\nbord$.
In particular, if the domain is Lipschitz and its boundary is endowed with the usual surface measure $\sigma$, then $\ddn {G_k(x-\cdot)}\big|_{\bord}\in L^2(\bord,\sigma)$ and it holds, for $f\in\Hp$ and a.e. $x\in\nbord$,
\begin{equation*}
\Dcal_kf(x)=\int_{\bord}\frac{\partial G_k}{\partial\nu_y}(x-y)f(y)\,\sigma(\dy).
\end{equation*}
Thus, the layer potential operators defined variationally for bounded two-sided admissible domains in this section appear as a direct generalisation of the classical layer potential operators, defined by means of the fundamental solution.

\subsection{Boundary values of the layer potentials}\label{Subsec:NP-Op}

The spirit of the transmission problem~\eqref{Eq:Helm-Prob} is to determine the solution $u$ based on its behaviour upon crossing the boundary $\bord$.
For that matter, the problem focuses on the intertwinement of the interior and exterior problems.
Nonetheless, a transmission solution can always be thought of as the pairing of a solution to the interior problem and one to the exterior one.
In this section, we introduce the operators allowing this change in perspective.

\subsubsection{Neumann-Poincaré operators}

The layer potential operators $\Scal_k$ and $\Dcal_k$ allow to reconstruct the solution to the transmission problem~\eqref{Eq:Helm-Prob} from its jumps in trace and in normal derivative.
To bind this problem, set on the whole space, to problems set on either the interior or the exterior domains, we introduce operators allowing to recover the boundary values of the transmission solution.

\begin{definition}
Let $\Omega$ be a two-sided admissible domain and let $k>0$.
The Neumann-Poincaré operator associated to the Helmholtz transmission problem~\eqref{Eq:Helm-Prob} is defined as
\begin{align*}
\Kcal_k:\Hp&\longrightarrow\Hp\\
f&\longmapsto \frac12(\Tr^i_{\bord}+\Tr^e_{\bord})\Dcal_k f.
\end{align*} 
\end{definition}

As in the classical case~\cite[Theorem 3.1.2]{nedelec_acoustic_2011}, the adjoint of the Neumann-Poincaré operator is connected to the boundary values of the single layer potential operator.

\begin{proposition}\label{Prop:K*}
Let $\Omega$ be a two-sided admissible domain and let $k>0$.
The adjoint of the Neumann-Poincaré operator $\Kcal_k$ is given by
\begin{align*}
\Kcal_k^*:\Hm&\longrightarrow\Hm\\
g&\longmapsto\frac12\left(\ddny{}i+\ddny{}e\right)\Scal_k g.
\end{align*}
\end{proposition}

\begin{proof}
Let $f\in\Hp$ and $g\in\Hm$. By definition of $\Kcal_k$ and of $\Scal_k$ (Definition~\ref{Def:SLP-Helm}), it holds (all the duality pairings are on $(\Hm,\Hp)$)
\begin{equation*}
2\langle g,\Kcal_k f\rangle=\left\langle\left\llbracket\ddn{}\Scal_k g\right\rrbracket,(\Tr_{\bord}^i+\Tr_{\bord}^e)\Dcal_kf\right\rangle.
\end{equation*}
On $\Omega$, Green's formula yields
\begin{align*}
\left\langle\ddny{}i\Scal_kg,\Tr_{\bord}^i(\Dcal_kf)\right\rangle&=\int_\Omega\nabla(\Scal_kg)\cdot\nabla(\overline{\Dcal_kf})\,\dx-k^2\int_\Omega(\Scal_kg)(\overline{\Dcal_kf})\,\dx\\
&=\overline{\left\langle\ddn{} \Dcal_kf,\Tr(\Scal_kg)\right\rangle}.
\end{align*}
On $\overline{\Omega}^c$, the non-integrability of the layer potentials at infinity prevents from performing the same computations directly.
Instead, we use a truncation once again.
If $B$ is a large ball containing $\Omega$ and $U\coloneq B\backslash\overline\Omega$, then -- introducing $\chi\in\mathscr{D}(\R^n)$ such that $\chi|_U=1$ -- it holds
\begin{multline*}
\left\langle\ddny{}e\Scal_kg,\Tr_{\bord}^e(\Dcal_kf)\right\rangle=\int_U\nabla(\Scal_kg)\cdot\nabla(\overline{\Dcal_kf})\,\dx-k^2\int_U(\Scal_kg)(\overline{\Dcal_kf})\,\dx\\
-\int_{\partial B}\tilde\Lambda_{\partial B}(\Tr_{\partial B}(\Scal_kg))\Tr_{\partial B}(\overline{\Dcal_kf})\dsig.
\end{multline*}
where $\tilde\Lambda_{\partial B}$ was defined by~\eqref{Eq:Lambda_tilde}, see Remark~\ref{Rem:PS_B'-L2}.
Note that by elliptic regularity, the layer potentials are smooth away from $\bord$ so that, in particular, their traces on $\partial B$ are in $L^2(\partial B,\sigma)$.
Since Proposition~\ref{Prop:K*} is known to hold in the classical case (see~\cite{nedelec_acoustic_2011} for instance), we can deduce
\begin{equation}\label{Eq:Pairing_LPs}
\int_{\partial B}\tilde\Lambda_{\partial B}(\Tr_{\partial B}(\Scal_kg))\Tr_{\partial B}(\overline{\Dcal_kf})\dsig=\int_{\partial B}\overline{\tilde\Lambda_{\partial B}(\Tr_{\partial B}(\Dcal_kf))}\Tr_{\partial B}(\Scal_kg)\dsig.
\end{equation}
Indeed, let $\hat\Omega$ be a smoother domain lying between $\Omega$ and $B$ for which $-k^2$ is neither a Dirichlet nor Neumann interior eigenvalue (a hypercube of suitable size for instance).
Then, Proposition~\ref{Prop:K*} holds when the domain is $\hat\Omega$.
Denoting by $\hat\Scal_k$ and $\hat\Dcal_k$ the layer potential operators for the Helmholtz transmission problem across the boundary of $\hat\Omega$, the same computations as above yield, for all $\hat f\in\Bcal(\partial\hat\Omega)$ and $\hat g\in\Bcal'(\partial\hat\Omega)$,
\begin{equation*}
\int_{\partial B}\tilde\Lambda_{\partial B}(\Tr_{\partial B}(\hat\Scal_k\hat g))\Tr_{\partial B}(\overline{\hat\Dcal_k\hat f})\dsig=\int_{\partial B}\overline{\tilde\Lambda_{\partial B}(\Tr_{\partial B}(\hat\Dcal_k\hat f))}\Tr_{\partial B}(\hat\Scal_k\hat g)\dsig.
\end{equation*}
By hypothesis on $\hat\Omega$ (see the proof of Proposition~\ref{Prop:IE-D-N} for details), for all $f\in\Hp$ and $g\in\Hm$, there exist $\hat f\in\Bcal(\partial\hat\Omega)$ and $\hat g\in\Bcal'(\partial\hat\Omega)$ such that $\Tr_{\partial B}(\Scal_k g)=\Tr_{\partial B}(\hat\Scal_k\hat g)$ and $\Tr_{\partial B}(\Dcal_k g)=\Tr_{\partial B}(\hat\Dcal_k\hat g)$ and~\eqref{Eq:Pairing_LPs} follows.
Hence,
\begin{equation*}
\left\langle\ddny{}e\Scal_kg,\Tr_{\bord}^e(\Dcal_kf)\right\rangle=\overline{\left\langle\ddn{} \Dcal_kf,\Tr(\Scal_kg)\right\rangle},
\end{equation*}
which leads to
\begin{align*}
2\left\langle g,\Kcal_k f\right\rangle&=\left\langle\ddny{}i\Scal_kg,\Tr_{\bord}^e(\Dcal_kf)\right\rangle-\left\langle\ddny{}e\Scal_kg,\Tr_{\bord}^i(\Dcal_kf)\right\rangle\\
&=-2\left\langle\left(\ddny{}i+\ddny{}e\right)\Scal_kg,\llbracket\Tr(\Dcal_kf)\rrbracket\right\rangle\\
&=2\left\langle\left(\ddny{}i+\ddny{}e\right)\Scal_kg,f\right\rangle,
\end{align*}
by the same computations.
\end{proof}

Given the definition of the operator $\Kcal_k$ and the characterisation of its adjoint $\Kcal_k^*$, the traces of the double layer potential can be expressed as
\begin{align*}
\Tr^i\circ\Dcal_k=-\frac12I+\Kcal_k&&\mbox{and}&&\Tr^e\circ\Dcal_k=\frac12I+\Kcal_k,
\end{align*}
since $\llbracket\Tr \Dcal_k\rrbracket=-I$, and the normal derivatives of the single layer potential as
\begin{align*}
\ddny{}i\circ\Scal_k=\frac12I+\Kcal_k^*&&\mbox{and}&&\ddny{}e\circ\Scal_k=-\frac12I+\Kcal_k^*,
\end{align*}
since $\llbracket\ddn{}\Scal_k\rrbracket=I$.
Seeking the solutions to the interior and exterior Dirichlet (resp. Neumann) problems for the Helmholtz equation in terms of a double (resp. single) layer potential allows to prove the following result on the invertibility of those operators.

\begin{proposition}\label{Prop:IE-D-N}
Let $\Omega$ be a two-sided admissible domain of $\R^n$ and let $k>0$.
\begin{enumerate}
\item[(i)] If $k^2$ is not an interior Dirichlet eigenvalue of $-\Delta$, then for $h\in\Hp$, the boundary equation
\begin{equation*}
\left(-\frac12I+\Kcal_k\right)f=h,
\end{equation*}
has a unique solution $f\in\Hp$, while for $\varphi\in\Hm$,
\begin{equation*}
\left(-\frac12I+\Kcal_k^*\right)g=\varphi,
\end{equation*}
has a unique solution $g\in\Hm$.
\item[(ii)] If $k^2$ is not an interior Neumann eigenvalue of $-\Delta$, then for $h\in\Hp$, the boundary equation
\begin{equation*}
\left(\frac12I+\Kcal_k\right)f=h,
\end{equation*}
has a unique solution $f\in\Hp$, while for $\varphi\in\Hm$,
\begin{equation*}
\left(\frac12I+\Kcal_k^*\right)g=\varphi,
\end{equation*}
has a unique solution $g\in\Hm$.
\end{enumerate}
\end{proposition}

\begin{proof}
Let us prove Point (i).
Let $h\in\Hp$.
By hypothesis on $k^2$, there exists a unique $u\in H^1(\Omega)$ such that $(\Delta+k^2)u=0$ on $\Omega$ and $\Tr^iu=h$.
Since the exterior Neumann problem is well-posed (Remark~\ref{Rem:Helm_Ext_WP}), there exists a unique $f\in\Hp$ such that $u=\Dcal_kf|_\Omega$.
It follows that $(-\frac12I+\Kcal_k)f=h$.
If $\tilde f\in\Hp$ satisfies $(-\frac12I+\Kcal_k)\tilde f=h$, then $\Dcal_k\tilde f|_\Omega$ solves the interior Dirichlet problem.
Hence $\Dcal_k\tilde f|_\Omega=u$, and $f=\tilde f$.
The second part of the statement follows by Banach's closed range theorem.
The proof of Point (ii) is similar, considering a single layer potential instead.
\end{proof}

\subsubsection{Calder\'on projectors}
The Neumann-Poincaré operators allow to recover the normal derivatives of the single layer potential and the traces of the double layer potential.
To retrieve the remaining boundary values of the layer potentials, consider the following operators
\begin{align*}
\Vcal_k&\coloneq\Tr\circ\Scal_k:\Hm\longrightarrow\Hp,\\
\Wcal_k&\coloneq-\ddn{}\circ\Dcal_k:\Hp\longrightarrow\Hm.
\end{align*}
By Green's third formula (Theorem~\ref{Th:Green3}), the matrix operator
\begin{equation*}
\Mcal_k\coloneq
\begin{pmatrix}
-\Kcal_k & \Vcal_k \\
\Wcal_k & \Kcal_k^*
\end{pmatrix}
:\Hp\times\Hm\longrightarrow\Hp\times\Hm,
\end{equation*}
allows to recover all the boundary values of a general Helmholtz transmission solution.
In particular, we can define two operators -- known as \emph{Calder\'on projectors} -- which, given jumps in trace and normal derivative, respectively return the interior and exterior boundary values of the associated transmission solution.

\begin{proposition}\label{Prop:Calderon-Proj}
Let $\Omega$ be a two-sided admissible domain of $\R^n$ and let $k>0$.
The operators
\begin{align}
&\Ccal_k^i\coloneq\frac12 I+\Mcal_k &&\mbox{and}&&\Ccal_k^e\coloneq\frac12 I-\Mcal_k,
\end{align}
acting as
\begin{align*}
\Ccal_k^i\begin{pmatrix}\llbracket \Tr u\rrbracket\\\llbracket\ddn u\rrbracket\end{pmatrix}=\begin{pmatrix}\Tr^iu\\\ddny ui\end{pmatrix}
&& \mbox{and}
&& \Ccal_k^e\begin{pmatrix}\llbracket \Tr u\rrbracket\\\llbracket\ddn u\rrbracket\end{pmatrix}=-\begin{pmatrix}\Tr^eu\\\ddny ue\end{pmatrix},
\end{align*}
where $u\in H^1_{\mathrm{loc_b}}(\nbord)$ satisfies $(\Delta+k^2)u=0$ on $\nbord$ and the $\mathrm{SRC}$~\eqref{Eq:SRC}, are linear and continuous projectors.
\end{proposition}

\begin{proof}
The fact that $\Ccal_k^i$ maps the jumps of a transmission solution to its interior boundary values directly follows from the definitions of the operators involved.
If $u\in H^1_{\mathrm{loc_b}}(\nbord)$ is a solution to the Helmholtz transmission problem, then $v\coloneq u\mathds{1}_\Omega\in H^1_{\mathrm{loc_b}}(\nbord)$ is also a solution to the Helmholtz transmission problem with $\llbracket\Tr v\rrbracket=\Tr_iu$ and $\llbracket\ddn v\rrbracket=\ddny ui$.
It follows that $\Ccal_k^i$ is a projector.
The proof for $\Ccal_k^e$ is similar.
\end{proof}
This allows to deduce the so-called \emph{Calder\'on relations} for the Helmholtz transmission problem (see~\cite[Theorem 3.1.3]{nedelec_acoustic_2011} in the classical case):
\begin{align*}
&\begin{cases}
\Kcal_k\Vcal_k=\Vcal_k\Kcal_k^*,\\
\Wcal_k\Kcal_k=\Kcal_k^*\Wcal_k,
\end{cases}
&&\mbox{and}
&&\begin{cases}
\Kcal_k^2+\Vcal_k\Wcal_k=\frac14I,\\
(\Kcal_k^*)^2+\Wcal_k\Vcal_k=\frac14I.
\end{cases}
\end{align*}

\section{Scattering problem}\label{Sec:Scattering}

The Calder\'on projectors from Theorem~\ref{Prop:Calderon-Proj} allow to view the transmission problem as the pairing of an interior and an exterior problem, with Dirichlet or Neumann boundary conditions.
In this section we use the analysis carried out on the transmission problem to study the exterior problem, also known as scattering problem, with Robin (impedance) boundary condition.

\subsection{Intrinsic impedance boundary condition}\label{Subsec:Robin}

Let $\Omega$ be a two-sided admissible domain of $\R^n$, let $k>0$ and $L\in\mathcal{L}(\Hp)$ playing the part of the impedance coefficient.
Consider the following impedance boundary value problem
\begin{equation}\label{Eq:Rob-Prob}
\begin{cases}
(\Delta+k^2)u=0 &\mbox{on }\overline{\Omega}^c,\\
\displaystyle\ddny ue\Big|_{\bord}+L\Tr^e_{\bord}u=h,\\
\mathrm{SRC},
\end{cases}
\end{equation}
where $h\in\Hm$. Due to the absence of a specified boundary measure (hence, of an $L^2$ space on $\bord$), the Robin boundary condition is to be understood in the weak sense, as
\begin{equation}\label{Eq:BC-Rob}
\forall f\in\Hp,\quad \left\langle\ddny ue\Big|_{\bord},f\right\rangle_{\Bcal'\!,\,\Bcal(\bord)}+\left\langle L\Tr^e_{\bord}u,f\right\rangle_{\Bcal(\bord)}=\langle h,f\rangle_{\Bcal'\!,\,\Bcal(\bord)}.
\end{equation}
Bearing in mind that the Riesz isometry between $\Hp$ and $\Hm$ is given by the $1$-harmonic Poincaré-Steklov operator $\Lambda_{\mathrm i}$ from~\eqref{Eq:PS-i},~\eqref{Eq:BC-Rob} could also be written as
\begin{equation*}
\forall f\in\Hp,\quad \left\langle\ddny ue\Big|_{\bord},f\right\rangle_{\Bcal'\!,\,\Hp}+\left\langle \Lambda_{\mathrm i}L\Tr^e_{\bord}u,f\right\rangle_{\Bcal'\!,\,\Hp}=\langle h,f\rangle_{\Bcal'\!,\,\Hp}.
\end{equation*}
A special case is when there exists a Lipschitz-continuous $\lambda:\overline{\Omega}^c\to\C$ such that for $f\in\Hp$, $L f=\lambda|_{\bord} f$.
It is known (see~\cite[Lemma 3.4]{claret_poincare-steklov_2025}) that the regularity of $\lambda$ implies that for $v\in H^1_{\mathrm{loc_b}}(\overline{\Omega}^c)$, $\lambda\Tr^e_{\bord}v=\Tr^e_{\bord}(\lambda v)$, and the boundary condition $\ddn u+\lambda\Tr u$ is well-defined in the sense of~\eqref{Eq:BC-Rob}.

Still due to the non-integrability of the solutions to the Helmholtz equation at infinity, we perform a truncation as in Subsection~\ref{Subsec:LP} and study the problem on $U\coloneq\overline{\Omega}^c\cap B$ where $B$ is a large open ball containing $\overline{\Omega}$.
The associated variational formulation is given by
\begin{multline}\label{Eq:Intrinsic_FV}
\forall v\in H^1(U),\quad \int_{U}\nabla u\cdot\nabla \bar v\,\dx-k^2\int_{U}u\bar v\,\dx-\langle L \Tr^e_{\bord}u,\Tr^e_{\bord}v\rangle_{\Hp}\\
-\int_{\partial B}\tilde\Lambda_{\partial B}(\Tr_{\partial B}u)\Tr_{\partial B}\bar v\,\dsig=-\langle h,\Tr^e_{\bord}v\rangle_{\Bcal'\!,\,\Hp}.
\end{multline}

Following the method from~\cite[Theorem 3.37]{colton_integral_2013}, we give a sufficient condition for the uniqueness of the solutions to~\eqref{Eq:Rob-Prob}.

\begin{proposition}\label{Prop:Impedance-Uniqueness}
Let $\Omega$ be a two-sided admissible domain of $\R^n$, let $k>0$ and assume $L\in\mathcal{L}(\Hp)$ is such that $\mathrm{Im}(\langle L f,f\rangle_{\Hp})\ge0$ for $f\in\Hp$.
Then, $u=0$ is the unique solution to~\eqref{Eq:Rob-Prob} for $h=0$ in $H^1_{\mathrm{loc_b}}(\overline\Omega^c)$.
\end{proposition}

\begin{proof}
If $u\in H^1_{\mathrm{loc_b}}(\overline\Omega^c)$ is a solution to~\eqref{Eq:Rob-Prob} for $h=0$, then it holds
\begin{align*}
k\overline{\left\langle\ddny ue,\Tr^e_{\bord} u\right\rangle}_{\Bcal'\!,\,\Bcal(\bord)}=-k\langle\Tr^e_{\bord} u,L\Tr^e_{\bord} u\rangle_{\Hp}
\end{align*}
It follows that
\begin{equation*}
\mathrm{Im}\left(k\overline{\left\langle\ddny ue,\Tr^e_{\bord} u\right\rangle}_{\Bcal'\!,\,\Bcal(\bord)}\right)\ge0
\end{equation*}
From there, a counterpart of~\cite[Theorem 3.12]{colton_integral_2013} in the case $n\ge 2$ and the framework of the trace spaces holds (Green's formula still being valid in that case, see Theorem~\ref{Th:Green}), and yields $u=0$.
\end{proof}

\begin{remark}
The condition regarding the impedance $L$ arising in Proposition~\ref{Prop:Impedance-Uniqueness}, namely
\begin{equation*}
\mathrm{Im}(\langle L f,f\rangle_{\Hp})\ge0\quad\mathrm{for}\;f\in\Hp,
\end{equation*}
can be restated in terms of dissipative operators, see for instance~\cite{phillips_dissipative_1959}.
A linear operator $A:H\to H$ on a Hilbert space $H$ is said to be dissipative when it satisfies
\begin{equation*}
\forall x\in H,\quad \mathrm{Re}(\langle Ax,x \rangle_H)\le 0.
\end{equation*}
With that terminology, the condition reads `$\mathrm{i}L$ is dissipative'.
\end{remark}

Since the solutions to~\eqref{Eq:Rob-Prob} are unique when $\mathrm{i}L$ is dissipative, the problem is well-posed anytime the underlying operator is Fredholm (of index $0$).
Denote by $A:H^1(U)\to H^1(U)$ the operator associated to the variational formulation~\eqref{Eq:Intrinsic_FV}, that is, for $u,v\in H^1(U)$,
\begin{multline}\label{Eq:Rob_Var_Op}
\langle Au,v\rangle_{H^1(U)}\coloneq\int_{U}\nabla u\cdot\nabla \bar v\,\dx-k^2\int_{U}u\bar v\,\dx\\
-\langle L \Tr^e_{\bord}u,\Tr^e_{\bord}v\rangle_{\Hp}-\int_{\partial B}\tilde\Lambda_{\partial B}(\Tr_{\partial B}u)\Tr_{\partial B}\bar v\,\dsig.
\end{multline}
In general, the boundary term $\langle L \Tr^e_{\bord}u,\Tr^e_{\bord}v\rangle_{\Hp}$ is neither compact nor coercive on $H^1(U)$.
However, the well-posedness of the exterior Dirichlet problem (Remark~\ref{Rem:Helm_Ext_WP}) induces the existence of a constant $c(\Omega,k)>0$ such that for all $u\in V_k(U)$, where $V_k(U)$ is the closed subspace of $H^1(U)$ defined by
\begin{equation}\label{Eq:V_k}
V_k(U)\coloneq\{v\in H^1(U)\;|\;(\Delta+k^2)v=0\mbox{ weakly on }U\},
\end{equation}
it holds
\begin{equation}\label{Eq:Trace-Estimates-Vk}
c(\Omega,k)\|u\|_{H^1(U)}\le\|\Tr^e_{\bord}u\|_{\Hp}\le \|\mathrm E_\Omega\|\|u\|_{H^1(U)},
\end{equation}
where $\mathrm E_\Omega:H^1(\Omega)\to H^1(\Omega)\oplus V_{\mathrm i}(\overline{\Omega}^c)$ denotes the $1$-harmonic extension from $\Omega$ to $\R^n$, such that for $w\in H^1(\Omega)$, $(\mathrm E_\Omega w)|_\Omega=w$ and $(-\Delta+1)\mathrm E_\Omega w=0$ on $\overline{\Omega}^c$.
The second inequality follows from the trace theorem (Theorem~\ref{Th:Trace}) and the equivalence between the interior and exterior trace norms~\cite[Proposition 3.4]{claret_convergence_2025}.
Since $H^1(U)\hookrightarrow L^2(U)$ is compact, the operator $A$ is Fredholm (of index $0$) on $V_k(U)$ anytime
\begin{equation*}
w\in V_k(U)\longmapsto\|w\|^2_{H^1(U)}-\langle L\Tr^e_{\bord}w,\Tr^e_{\bord}w\rangle_{\Hp}-\int_{\partial B}\tilde\Lambda_{\partial B}(\Tr_{\partial B}w)\Tr_{\partial B}\bar w\,\dsig
\end{equation*}
is coercive (up to a sign).
By~\eqref{Eq:Ineq-Re-PS} and~\eqref{Eq:Trace-Estimates-Vk}, this holds anytime
\begin{multline}\label{Eq:Coercive-condition}
\|L\|_{\mathcal{L}(\Hp)}<\|\mathrm E_\Omega\|^{-2}\qquad\mbox{or}\qquad\\\inf_{\|w\|_{H^1(U)}=1}\mathrm{Re}\left(\langle L\Tr^e_{\bord}w,\Tr^e_{\bord}w\rangle_{\Hp}+\int_{\partial B}\tilde\Lambda_{\partial B}(\Tr_{\partial B}w)\Tr_{\partial B}\bar w\,\dsig\right)>1.
\end{multline}
Denoting, for $k>0$,
\begin{equation}\label{Eq:Set-Impedances}
\mathscr{L}_k\coloneq\big\{L\in\mathcal{L}(\Hp)\;\big|\;\mathrm{i}L\mbox{ is dissipative}\quad\mbox{and}\quad\eqref{Eq:Coercive-condition}\big\},
\end{equation}
the (nonempty) set of suitable impedances, we state the well-posedness of the Robin problem as follows, reversing the truncation as in Lemma~\ref{Lem:Truncation}.

\begin{proposition}\label{Prop:Rob-WP}
Let $\Omega$ be a bounded two-sided admissible domain of $\R^n$ and let $k>0$.
If $L\in\mathscr{L}_k$, then, for all $h\in\Hm$, the impedance problem~\eqref{Eq:Rob-Prob} is well-posed on $H^1_{\locb}(\overline{\Omega}^c)$.
\end{proposition}

Seeking the solution to~\eqref{Eq:Rob-Prob} in the form of a single or a double layer potential, the well-posedness of~\eqref{Eq:Rob-Prob} implies that of the following boundary equations, in the spirit of~\cite{colton_inverse_2013,brakhage_uber_1965,chandler-wilde_sobolev_2017} for instance.

\begin{theorem}\label{Th:Boundary-Integral-Equation}
Let $\Omega$ be a two-sided admissible domain of $\R^n$, let $k>0$ and $L\in\mathscr{L}_k$.
\begin{enumerate}
\item[(i)] If $k^2$ is not an interior Dirichlet eigenvalue of $-\Delta$, then, for all $h\in\Hm$, the boundary equation
\begin{equation*}
\left(-\frac12I+\Kcal_k^*+\Lambda_{\mathrm i}L\Vcal_k\right)g=h,
\end{equation*}
has a unique solution $g\in\Hm$, where $\Lambda_{\mathrm i}$ is defined by~\eqref{Eq:PS-i};
\item[(ii)] If $k^2$ is not an interior Neumann eigenvalue of $-\Delta$, then, for all $h\in\Hm$, the boundary equation
\begin{equation*}
\left(-\Wcal_k+\Lambda_{\mathrm i}L\left(\frac12I+\Kcal_k\right)\right)f=h,
\end{equation*}
has a unique solution $f\in\Hp$.
\end{enumerate}
\end{theorem}

\begin{proof}
Let us prove Point (i).
Let $h\in\Hm$.
Since~\eqref{Eq:Rob-Prob} is well-posed, there exists a unique solution $u\in H^1_{\mathrm{loc_b}}(\overline{\Omega}^c)$ associated to $h$.
By hypothesis on $k^2$, there exists a unique $g\in\Hm$ such that $u=\Scal_k g|_{\overline\Omega^c}$.
By definition of the operators involved, $g$ solves the boundary equation.
If $\varphi$ solves the boundary equation, then $\Scal_k\varphi|_{\overline\Omega^c}$ solves the impedance problem~\eqref{Eq:Rob-Prob}, and is therefore equal to $u$, hence $\varphi=g$.
The proof of Point (ii) is similar, considering a double layer potential instead.
\end{proof}

The assumptions on the impedance to ensure the well-posedness of~\eqref{Eq:Rob-Prob} in Proposition~\ref{Prop:Rob-WP}, namely that $L\in\mathscr{L}_k$, differ from the hypotheses in the classical case.
Indeed, it is well-known that when the boundary is smooth (or at least Lipschitz) and endowed with the surface measure $\sigma$, the exterior Helmholtz problem with boundary condition $\ddny ue+\lambda\Tr^eu=h$ for $\lambda\in L^\infty(\bord,\sigma)$ and $h\in L^2(\bord,\sigma)$ is well-posed on $H^1_{\locb}(\overline{\Omega}^c)$ anytime $\mathrm{Im}(\bar k\lambda)\ge0$, see for instance~\cite[§~3.7]{colton_integral_2013}.
This is due to the fact that the boundary condition is not understood in the same space: in our framework, it is understood in the space $\Hm$, that is, in the sense of~\eqref{Eq:BC-Rob}.
In the classical case however, the boundary condition is understood in the sense of $L^2(\bord,\sigma)$, that is, in terms of boundary integrals:
\begin{equation*}
\forall f\in L^2(\bord,\sigma),\quad \int_{\bord}\ddny ue\bar f\,\dsig+\int_{\bord}\lambda\Tr^e u\bar f\,\dsig=\int_{\bord}h\bar f\,\dsig.
\end{equation*}
Bearing in mind the embedding $L^2(\bord,\sigma)\hookrightarrow\Hm$, in the sense of~\eqref{Eq:Incl-L2-B'}, this condition can be rewritten as
\begin{equation*}
\forall f\in\Hp,\quad \left\langle\ddny ue\Big|_{\bord},f\right\rangle_{\Bcal'\!,\,\Bcal(\bord)}+\left\langle L\Tr^e_{\bord}u,f\right\rangle_{L^2(\bord,\sigma)}=\langle h,f\rangle_{\Bcal'\!,\,\Bcal(\bord)}.
\end{equation*}
While the terms involving duality pairings between $\Hp$ and $\Hm$ (that is, up to equivalent norms, $H^{\frac 12}(\bord,\sigma)$ and $H^{-\frac 12}(\bord,\sigma)$ in this case~\cite[Chapter IV, Appendix]{dautray_mathematical_1988-2}) reduce to a standard inner product on $L^2(\bord,\sigma)$, the inner product on $\Hp$ should be written in terms of an inner product on $H^{\frac 12}(\bord,\sigma)$ instead.
For that matter, understanding the Robin boundary condition in terms of~\eqref{Eq:BC-Rob}, which is natural in the absence of a boundary measure, is not a strict generalisation of the Robin problem as it is usually understood, in the sense of $L^2(\bord,\sigma)$.

In light of this observation, one may question the relevance of setting the impedance problem in our framework, rather than the usual one.
As it was mentioned before, understanding the boundary condition in the sense of~\eqref{Eq:BC-Rob} is natural in the framework of extension domains for there is no specified measure on the boundary, hence no boundary $L^2$ space.
To consider the impedance problem in the usual $L^2$-sense means further assuming the boundary of the domain is endowed with a `suitable' measure (so that $\Hp\hookrightarrow L^2(\bord,\mu)$ be compact), while -- to the best of our knowledge -- the impact the existence of such a measure has on the boundary regularity is not known.
Nonetheless, we may wonder if our framework is physically relevant for the underlying problem.
To that end, consider the case of a constant impedance $L=\lambda\in\C$.
Then, the homogeneous problem is uniquely solved anytime $\mathrm{Im}(\bar k\lambda)\ge0$.

In addition, the problematic boundary term can be expressed as
\begin{equation*}
\langle L\Tr^e_{\bord}u,\Tr^e_{\bord}u\rangle_{\Hp}=\lambda\|\Tr^e_{\bord}u\|_{\Hp}^2,
\end{equation*}
so that by~\eqref{Eq:Trace-Estimates-Vk}, the operator $A$ from~\eqref{Eq:Rob_Var_Op} is Fredholm (of index $0$) on $V_k(U)$ (see~\ref{Eq:V_k}) anytime
\begin{equation}\label{Eq:Impedance_Constant}
\mathrm{Re}(\lambda)<\|\mathrm E_\Omega\|^{-2}\qquad\mbox{or}\qquad c(\Omega,k)^{-2}\left(1-\mathrm{Re}\left(\int_{\partial B}\tilde\Lambda_{\partial B}(\Tr_{\partial B}w\right)\right)<\mathrm{Re}(\lambda).
\end{equation}
However, given our conventions, that last condition is always satisfied by physical systems.
Indeed, consider a wave $w:(t,x)\mapsto u(x)e^{-\mathrm i\omega t}$ satisfying
\begin{equation*}
\begin{cases}
\partial_t^2w-c^2\Delta w=0 &\mbox{on }]0,+\infty[\times\overline{\Omega}^c,\\
\ddny we -\frac1\omega\mathrm{Im}(\lambda)\partial_tw + \mathrm{Re}(\lambda)w=he^{-\mathrm i\omega t} &\mbox{on }]0,+\infty[\times\bord,\\
w(0,\cdot)=w_0\quad\mbox{and}\quad \partial_t w(0,\cdot)=v_0 &\mbox{on }\overline{\Omega}^c,
\end{cases}
\end{equation*}
for some smooth initial conditions $(w_0,v_0)$, where $c>0$ is the wave velocity and $\omega>0$ is its angular frequency, so that $k=\frac\omega c>0$ and the problem is compatible with~\eqref{Eq:Rob-Prob} upon eliminating the time-dependency~\cite{bardos_variational_1994}.
Then, formally (multiplying by $\partial_tw$ and integrating over $U\coloneq\overline{\Omega}^c\cap B$ with $B$ a large ball), the time variation of the energy of the system is governed by
\begin{multline*}
\frac{\mathrm d}{\mathrm dt}\left(\frac12\int_U\left(|\partial_tw|^2+c^2|\nabla w|^2\right)\,\dx-\frac {c^2}2\mathrm{Re}(\lambda)\|\Tr^ew\|_{\Hp}^2\right)=\Big\langle\ddny wi,\Tr^i(\partial_t w)\Big\rangle_{\Bcal'\!,\,\Bcal(\partial B)}\\
-c^2\left(\langle he^{-\mathrm{i}\omega t},\Tr^e(\partial_t w)\rangle_{\Bcal'\!,\,\Hp}-\frac1\omega\mathrm{Im}(\lambda)\|\Tr^e(\partial_tw)\|_{\Hp}^2\right).
\end{multline*}
In particular, the additional energy term generated at $\bord$, $-\frac{c^2}2\mathrm{Re}(\lambda)\|\Tr^ew\|_{\Hp}^2$, corresponds to the energy reflection at the boundary, so that physically, $\mathrm{Re}(\lambda)\le0$, thus $\mathrm{Re}(\lambda)<\|\mathrm E_\Omega\|^{-2}$ and~\eqref{Eq:Impedance_Constant} holds.

\begin{remark}
The sign difference between our condition $\mathrm{Re}(\lambda)\le0$ and the classical case~\cite{bardos_variational_1994} is due to our sign convention in the definition of the exterior normal derivative (Theorem~\ref{Th:Green}).
\end{remark}

\subsection{Impedance optimisation}\label{Subsec:Control}

The analysis carried out on the impedance problem~\eqref{Eq:Rob-Prob} and its connection to the layer potential operators (Theorem~\ref{Th:Boundary-Integral-Equation}) allow to generalise results found in~\cite[Chapter 9]{colton_integral_2013} regarding the optimisation of a planar scattered wave (see also~\cite{kirsch_optimal_1981,angell_generalized_1982}) to the case of a two-sided admissible obstacle in an $n$-dimensional setting.
Let $\Omega$ be a two-sided admissible domain of $\R^n$ and let $k>0$ be such that $k^2$ is not an interior Dirichlet eigenvalue for $-\Delta$.
Consider an incident field $u^i\in H^1_{\mathrm{loc_b}}(\overline{\Omega}^c)$ such that $(\Delta+k^2)u^i=0$ on $\overline{\Omega}^c$.
Given an impedance $L\in\mathscr{L}_k$, assume this wave is scattered by the obstacle $\Omega$, thus generating a scattered field $u_L\in H^1_{\mathrm{loc_b}}(\overline{\Omega}^c)$ such that the total field $u^t_L\coloneq u^i+u_L$ satisfies
\begin{equation*}
\begin{cases}
(\Delta+k^2)u^t_L=0 &\mbox{on }\overline{\Omega}^c,\\
\displaystyle\ddny {u^t_L}e\Big|_{\bord}+L\Tr^e_{\bord}u^t_L=0.
\end{cases}
\end{equation*}
Stemming from the obstacle, the scattered field satisfies the $\mathrm{SRC}$~\eqref{Eq:SRC}, hence it solves
\begin{equation*}
\begin{cases}
(\Delta+k^2)u_L=0 &\mbox{on }\overline{\Omega}^c,\\
\displaystyle\ddny {u_L}e\Big|_{\bord}+L\Tr^e_{\bord}u_L=-\ddny {u^i}e\Big|_{\bord}-L\Tr^e_{\bord}u^i,\\
\mathrm{\mathrm{SRC}}.
\end{cases}
\end{equation*}
Given a measurable set $\Theta$ of the Euclidean unit sphere $\mathbb{S}^{n-1}$, the far-field power of the scattered field through $\Theta$ can be expressed as
\begin{equation*}
 Q_\Theta(L)\coloneq\lim_{r\to\infty}\int_{\mathbb S^{n-1}}\mathds{1}_\Theta(\hat x)|u_L(r\hat x)|^2\,\sigma(\mathrm d\hat x),
\end{equation*}
where $\sigma$ is the surface measure on the unit sphere $\mathbb S^{n-1}$. The setting is represented in Figure~\ref{Fig:Optim}.

\begin{figure}[ht]
\centering
\begin{tikzpicture}
\draw (4,0) node{\includegraphics[scale=0.7]{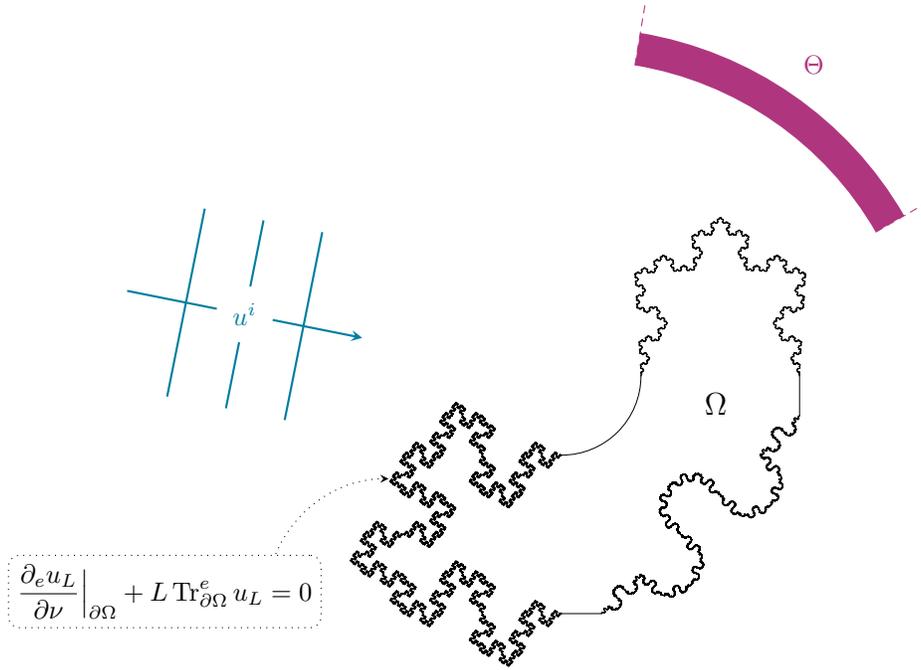}};
\draw (5.825,0.5) node {\large $\Omega$} ;
\draw (-1.5,-2) node [rectangle, draw, dotted, rounded corners=0.15cm] {$\displaystyle\ddny {u^t_L}e\Big|_{\bord}+L\Tr^e_{\bord}u^t_L=0$} ;
\draw[->, >=stealth, dotted] (-0.025,-1.5) arc (157.38:90:1.625) ;
\draw[PineGreen!40!RoyalBlue, thick] (-2+3.125*0.5-1*0.25, 2-0.625*0.5-5*0.25) -- ++(0.25*2,5*0.25*2) ;
\draw[PineGreen!40!RoyalBlue, thick] (-2+3.125*0.75-1*0.25, 2-0.625*0.75-5*0.25) -- ++(0.25*2,5*0.25*2) ;
\draw[PineGreen!40!RoyalBlue, thick] (-2+3.125*0.25-1*0.25, 2-0.625*0.25-5*0.25) -- ++(0.25*2,5*0.25*2) ;
\draw[->, >=stealth, PineGreen!40!RoyalBlue, thick] (-2,2) -- ++(3.125,-0.625) node[midway, circle, fill=white]{$u^i$} ;
\draw[SeaGreen!80!black, densely dashed] (4.75,5) -- (4.894,5.864) ;
\draw[SeaGreen!80!black, densely dashed] (7.951,2.781) -- (8.709,3.219) ;
\fill[SeaGreen!80!black] (4.75,5) arc (80.538:30:4.562) -- (8.33,3) arc (30:80.538:5) ;
\draw[SeaGreen!80!black] (7.12715,5.01085) node {$\Theta$} ;
\end{tikzpicture}
\caption{Illustration of the setting of the optimisation problem.}
\label{Fig:Optim}
\end{figure}
The optimisation problem consists in finding an impedance $L^*\in\mathscr{L}_k^0$ which maximises the far-field power through $\Theta$ over a certain range of impedances $\mathscr{L}_k^0\subset\mathscr{L}_k$.

To gain a better understanding of the quantity $ Q_\Theta(L)$ for $L\in\mathscr{L}_k$, we derive the asymptotic expansion of $u_L(x)$ as $|x|\to\infty$.
By hypothesis on $k$ (see also Theorem~\ref{Th:Boundary-Integral-Equation}), $u_L$ can be thought of as a single layer potential.
Thus, Proposition~\ref{Prop:SLP-Repr} (see also~\eqref{Eq:Link_Sk_Gk}) allows to write the scattered wave as
\begin{equation}\label{Eq:Scattered_field}
u_L(x)=\left\langle\Big\llbracket\ddn{u_L}\Big\rrbracket,\Tr_{\bord}\overline{G_k(x-\cdot)}\right\rangle_{\Bcal',\Hp}.
\end{equation}
For $y\in\bord$, denoting $x=|x|\hat x\in\R^n\backslash\{0\}$, the asymptotic expansion
\begin{equation*}
|x-y|=|x|-\hat x\cdot y+\underset{|x|\to\infty}{O}(|x|^{-1})
\end{equation*}
as well as that of the Hankel function~\cite[Eq. (10.17.5)]{nist_digital_nodate}
\begin{equation*}
H^{(1)}_{\frac n2-1}(z)=\sqrt{\frac2\pi}z^{-\frac12}e^{\mathrm i(z-(n-1)\frac\pi4)}(1+\underset{z\to\infty}{O}(z^{-1}))
\end{equation*}
allow to expand the Newtonian kernel~\eqref{Eq:Hankel} as
\begin{equation}\label{Eq:Newt-Asymptotic}
G_k(x-y)=\frac{e^{\mathrm ik|x|}}{|x|^{\frac{n-1}2}}F\left(\hat x,y\right)+\underset{|x|\to\infty}{O}(|x|^{-\frac{n+1}2}),
\end{equation}
where
\begin{equation*}
F\left(\hat x,y\right)\coloneq\frac 1{4\pi}e^{-\mathrm i(n-3)\frac\pi4}\left(\frac k{2\pi}\right)^{\frac{n-3}2}\exp\big(-\mathrm ik\hat x\cdot y\big),
\end{equation*}
is the far-field pattern of the kernel.
Let $K\subset\R^n$ be a compact set such that $\bord\subset \mathring K$ and let $B$ be a large open ball such that $K\subset B$.
Then, there exists $C>0$ such that the remainder in~\eqref{Eq:Newt-Asymptotic} can be expressed as
\begin{equation}\label{Eq:Remainder}
\underset{|x|\to\infty}{O}(|x|^{-\frac{n+1}2})=\varepsilon(x,y)|x|^{-\frac{n+1}2},
\end{equation}
where
\begin{equation*}
\forall x\in \overline{B}^c,\;\forall y\in K,\quad |\varepsilon(x,y)|\le C.
\end{equation*}
In the classical case,~\eqref{Eq:Scattered_field} can be written in terms of a boundary integral so that the asymptotic expansion of the solution $u_L$ can be directly derived from that of the kernel $G_k$.
Working on the trace space directly, that derivation is less straightforward.
Since~\eqref{Eq:Scattered_field} involves the kernel $G_k$ as an element of $H^1$ near $\bord$, we need a similar bound on the gradient of $\varepsilon$ to derive the asymptotic expansion of the scattered field $u_L$.

\begin{lemma}
Let $\Omega$ be a two-sided admissible domain of $\R^n$ and let $k>0$. It holds
\begin{equation*}
u_L(x)=\frac{e^{\mathrm ik|x|}}{|x|^{\frac{n-1}2}}\left\langle\Big\llbracket\ddn{u_L}\Big\rrbracket,\Tr \left(F\left({\hat x},\cdot\right)\right)\right\rangle_{\Bcal',\Hp}+\underset{|x|\to\infty}{O}(|x|^{-\frac{n+1}2}).
\end{equation*}
\end{lemma}

\begin{proof}
Let us estimate the gradient of the remainder~\eqref{Eq:Remainder}.
It holds
\begin{equation*}
\nabla_x\left(|x|e^{\mathrm ik|x|}F\left({\hat x},y\right)\right)=\mathrm ikxe^{\mathrm ik|x|}F\left({\hat x},y\right)+\underset{|x|\to\infty}{O}(1),
\end{equation*}
while
\begin{equation*}
\nabla_y\left(|x|e^{\mathrm ik|x|}F\left({\hat x},y\right)\right)=-\mathrm ikxe^{\mathrm ik|x|}F\left({\hat x},y\right)+\underset{|x|\to\infty}{O}(1).
\end{equation*}
From the asymptotic expansion of the derivative of the Hankel function $H^{(1)}_0$
\begin{equation*}
(H^{(1)}_0)'(z)=\sqrt{\frac 2{\pi z}}e^{\mathrm i(z+\frac\pi4)}+\underset{z\to\infty}{O}(z^{-\frac32}),
\end{equation*}
and noticing that
\begin{equation*}
\frac{x-y}{|x-y|}=\hat x+\underset{|x|\to\infty}{O}(|x|^{-1}),
\end{equation*}
we can deduce
\begin{equation*}
\nabla_x(|x|^{\frac{n+1}2}G_k(x-y))=\mathrm{i}kxe^{\mathrm ik|x|}F\left({\hat x},y\right)+\underset{|x|\to\infty}{O}(1),
\end{equation*}
and
\begin{equation*}
\nabla_y(|x|^{\frac{n+1}2}G_k(x-y))=-\mathrm{i}kxe^{\mathrm ik|x|}F\left({\hat x},y\right)+\underset{|x|\to\infty}{O}(1).
\end{equation*}
Altogether, we can deduce $|\nabla\varepsilon(x,y)|=\underset{|x|\to\infty}{O}(1)$, and there exists $C>0$ such that
\begin{equation*}
\forall x\in \overline{B}^c,\;\forall y\in K,\quad |\nabla\varepsilon(x,y)|\le C.
\end{equation*}
If follows that, for $\chi\in\mathscr{D}(\R^n)$ such that $\chi=1$ on a neighbourhood of $\bord$, and $\operatorname{supp}\chi\subset K$,
\begin{equation*}
\left\langle\Big\llbracket\ddn{u_L}\Big\rrbracket,\Tr\varepsilon(x,\cdot)\right\rangle\le\left\|\Big\llbracket\ddn{u_L}\Big\rrbracket\right\|_{\Hm}\|\varepsilon(x,\cdot)\chi\|_{H^1(K)}=\underset{r\to\infty}{O}(1),
\end{equation*}
and the result follows.
\end{proof}

Given the asymptotic expansion of the scattered field $u_L$, the far-field power through $\Theta$ can be rewritten as
\begin{equation*}
Q_\Theta(L)=\int_{\Theta}\bigg|\left\langle\Big\llbracket\ddn{u_L}\Big\rrbracket,\Tr (F(\theta,\cdot))\right\rangle_{\Bcal',\Hp}\bigg|^2\,\dth.
\end{equation*}

We prove the continuity of the far-field power $Q_\Theta$.
To do so, we recall the following standard result on the convergence of inverse operators.

\begin{lemma}\label{Lem:CV-inverse}
Let $B$ be a Banach space, and let $T\in\mathcal{L}(B)$ and $(T_m)_{m\in\N}\in\mathcal{L}(B)^\N$ be invertible operators.
If $T_m\to T$, then $T_m^{-1}\to T^{-1}$ in $\mathcal{L}(B)$ as well.
\end{lemma}

\begin{proof}
Since $B$ is a Banach space and $T\in\mathcal{L}(B)$, it holds $T^{-1}\in\mathcal{L}(B)$ and there exists $c>0$ such that
\begin{equation*}
c\|x\|_B\le \|Tx\|_B,\quad x\in B.
\end{equation*}
For $m\in\N$ large enough, $\|T-T_m\|_{\mathcal{L}(B)}\le \frac c2$, so that for $x\in B$,
\begin{equation*}
\|T_mx\|_B\ge \|Tx\|_B-\|(T-T_m)x\|_B\ge\frac c2 \|x\|_B.
\end{equation*}
It follows that $(T_m^{-1})_{m\in\N}$ is bounded in $\mathcal{L}(B)$, hence
\begin{equation*}
\|T^{-1}-T_m^{-1}\|_{\mathcal{L}(B)}\le\|T_m^{-1}\|_{\mathcal{L}(B)}\|T_m-T\|_{\mathcal{L}(B)}\|T^{-1}\|_{\mathcal{L}(B)}\xrightarrow[m\to\infty]{}0.\qedhere
\end{equation*}
\end{proof}

\begin{theorem}
Let $\Omega$ be a two-sided admissible domain of $\R^n$ and let $k>0$ be such that $k^2$ be outside the Dirichlet spectrum of $-\Delta$ on $\Omega$.
Let $\Theta$ be a measurable subset of $\mathbb S^{n-1}$.
Then the far-field power through $\Theta$, $Q_\Theta:\mathscr{L}_k\to\R$, is (sequentially) continuous.
\end{theorem}

\begin{proof}
Let $(L_m)_{m\in\N}\in \mathscr{L}_k^\N$ be a sequence converging to $L\in\mathscr{L}_k$ (in the sense of $\mathcal L(\Hp)$). By Theorem~\ref{Th:Boundary-Integral-Equation}, there exists a sequence $(g_m)\in\Hm^\N$ satisfying
\begin{equation*}
\forall m\in\N,\quad \left(-\frac12 I+\Kcal_k^*+\Lambda_{\mathrm i} L_m\Vcal_k\right)g_m=-\ddny{u^i}e\Big|_{\bord}-L_m\Tr^e_{\bord}u^i\in\Hm.
\end{equation*}
As a matter of fact, $g_m$ is none other than $\llbracket\ddn{u_{L_m}}\rrbracket$ for $m\in\N$.
It holds
\begin{equation*}
\left(-\frac12 I+\Kcal_k^*+\Lambda_{\mathrm i} L_m\Vcal_k\right)\longrightarrow \left(-\frac12 I+\Kcal_k^*+\Lambda_{\mathrm i} L_m\Vcal_k\right)\quad\mbox{in }\mathcal{L}(\Hm),
\end{equation*}
and all the operators involved (including the limit) are invertible.
By Lemma~\ref{Lem:CV-inverse}, it follows that the inverses converge as well in $\mathcal{L}(\Hm)$, hence $g_m\to g=\llbracket\ddn{u_L}\rrbracket$ in $\Hm$.
It follows that, for $\theta\in\mathbb{S}^{n-1}$,
\begin{equation*}
\left\langle g_m,\Tr (F(\theta,\cdot))\right\rangle_{\Bcal'\!,\,\Bcal(\bord)}\xrightarrow[m\to\infty]{} \left\langle g,\Tr (F(\theta,\cdot))\right\rangle_{\Bcal'\!,\,\Bcal(\bord)}.
\end{equation*}
Taking $0\in\Omega$, let $\chi\in\mathscr D(\R^2)$ be such that $0\not\in\operatorname{supp}\chi$ and $\chi=1$ near $\bord$.
Then, it holds
\begin{equation*}
\|\Tr (F(\theta,\cdot))\|_{\Hp}=\|\Tr (\chi F(\theta,\cdot))\|_{\Hp}\le\|\chi F(\theta,\cdot)\|_{H^1(\Omega)}.
\end{equation*}
Since $F\in C^\infty(\mathbb{S}^{n-1}\times\operatorname{supp}\chi)$, there exists $M>0$ such that
\begin{equation*}
\forall\theta\in\mathbb{S}^{n-1},\quad \|\chi F(\theta,\cdot)\|_{H^1(\Omega)}\le M.
\end{equation*}
By dominated convergence, it follows that
\begin{equation*}
Q_\Theta(L_m)\xrightarrow[m\to\infty]{}Q_\Theta(L).\qedhere
\end{equation*}
\end{proof}

The continuity of the form $Q_\Theta$ allows to state the existence of a solution to the optimisation problem for two-sided admissible obstacles in dimension $n\ge2$, thus generalising~\cite[Theorem 9.17]{colton_integral_2013}, see also~\cite{kirsch_optimal_1981,angell_generalized_1982}.

\begin{corollary}
Let $\Omega$ be a two-sided admissible domain of $\R^n$ and let $k>0$ be such that $k^2$ be outside the Dirichlet spectrum of $-\Delta$ on $\Omega$.
Let $\Theta$ be a measurable subset of $\mathbb{S}^{n-1}$.
Let $\mathscr{L}_k^0$ be a compact subset of $\mathscr{L}_k$.
Then, there exists a solution to the optimisation problem, that is
\begin{equation*}
\exists L^*\in\mathscr{L}_k^0,\quad \inf_{L\in\mathscr{L}_k^0}Q_\Theta(L)=Q_\Theta(L^*).
\end{equation*}
\end{corollary}

\section*{Acknowledgements}
The author is grateful to Anna Rozanova-Pierrat for her meaningful advice, David P. Hewett for his valuable feedback, Simon N. Chandler-Wilde for insightful discussions on the topic, as well as Marien Chenaud for his astute remarks.

\bibliographystyle{siam}
\def\refname{References}
\bibliography{BibGC.bib}

\newpage

\end{document}